\documentclass[notitlepage,11pt,reqno]{amsart}
\usepackage{amssymb,amsmath,amscd,amsthm,mathrsfs}
\usepackage[breaklinks=true]{hyperref}
\usepackage[usenames]{color}
\usepackage{enumerate,verbatim,soul}

\newcommand{\R}{\ensuremath{\mathbb R}}
\newcommand{\E}{\ensuremath{\mathbb E}}
\newcommand{\W}{\ensuremath{\mathcal W}}
\newcommand{\bp}{\ensuremath{\mathbb P}}
\newcommand{\G}{\ensuremath{\mathcal G}}
\newcommand{\boxnorm}[1]{\ensuremath{\| #1 \|_\Box}}
\newcommand{\mapdef}[1]{\ensuremath{\overset{#1}{\longrightarrow}}}
\def\tangle#1{\langle #1 \rangle}

\DeclareMathOperator{\Hom}{\ensuremath{Hom}}
\DeclareMathOperator*{\cart}{\times}

\newtheorem{theorem}[equation]{Theorem}
\newtheorem{prop}[equation]{Proposition}

\newtheorem{cor}[equation]{Corollary}
\newtheorem{lemma}[equation]{Lemma}
\newtheorem{utheorem}{\textrm{\textbf{Theorem}}}

\theoremstyle{definition}
\newtheorem{remark}[equation]{Remark}
\newtheorem{definition}[equation]{Definition}
\newtheorem{ex}[equation]{Example}

\numberwithin{equation}{section}

\usepackage[letterpaper]{geometry}
\begin{document}

\title{Model-free consistency of graph partitioning}

\author[Peter Diao]{Peter Diao$^1$}

\address{$^1$Statistical and Applied Mathematical Sciences Institute,
Research Triangle Park, NC, USA; email: {\tt peter.z.diao@gmail.com} (P.
Diao)}

\author[Dominique Guillot]{Dominique Guillot$^2$}
\address{$^2$Department of Mathematical Sciences, University of
Delaware, Newark, DE, USA;\hfill\break
email: {\tt dguillot@udel.edu}}

\author[Apoorva Khare]{Apoorva Khare$^3$}
\address{$^3$Departments of Mathematics and Statistics, Stanford
University, Stanford, CA, USA;
email: {\tt khare@stanford.edu} (A. Khare), 
{\tt brajaratnam01@gmail.com} (B. Rajaratnam)} 

\author[Bala Rajaratnam]{Bala Rajaratnam$^3$}

\date{\today}



\thanks{P.D., D.G., A.K., and B.R.~are partially supported by the
following: US Air Force Office of Scientific Research grant award
FA9550-13-1-0043, US National Science Foundation under grant DMS-0906392,
DMS-CMG 1025465, AGS-1003823, DMS-1106642, DMS-CAREER-1352656, Defense
Advanced Research Projects Agency DARPA YFA N66001-11-1-4131, the UPS
Foundation, SMC-DBNKY, an NSERC postdoctoral fellowship, and the
University of Delaware Research Foundation.}

\begin{abstract}
In this paper, we exploit the theory of dense graph limits to provide a
new framework to study the stability of graph partitioning methods, which
we call {\em structural consistency}. Both stability under perturbation
as well as asymptotic consistency (i.e., convergence with probability $1$
as the sample size goes to infinity under a fixed probability model)
follow from our notion of structural consistency. By formulating
structural consistency as a continuity result on the graphon space, we
obtain robust results that are completely independent of the data
generating mechanism.
In particular, our results apply in settings where observations are not
independent, thereby significantly generalizing the common probabilistic
approach where data are assumed to be i.i.d.

In order to make precise the notion of structural consistency of graph
partitioning, we begin by extending the theory of graph limits to include
vertex colored graphons. We then define \textit{continuous node-level
statistics} and prove that graph partitioning based on such statistics is
consistent. Finally, we derive the structural consistency of commonly
used clustering algorithms in a general model-free setting. These include
clustering based on local graph statistics such as homomorphism
densities, as well as the popular spectral clustering using the
normalized Laplacian.

We posit that proving the continuity of clustering algorithms in the
graph limit topology can stand on its own as a more robust form of
model-free consistency. We also believe that the mathematical framework
developed in this paper goes beyond the study of clustering algorithms,
and will guide the development of similar model-free frameworks
to analyze other procedures in the broader mathematical sciences.
\end{abstract}
\maketitle

\tableofcontents

\section{Introduction}

In this paper, we characterize continuity of certain maps from graphon
space to colored graphon space defined by popular graph clustering
algorithms such as spectral clustering.
Grouping or clustering objects according to their similarity is a
fundamental problem in many areas of modern science. The objective of
clustering is to identify such \textit{clusters} in data, where objects
assigned to the same cluster look roughly similar, whereas objects
belonging to different clusters are different. Various strategies have
been proposed to formulate and solve clustering problems in a rigorous
way (see e.g.~\cite{cluster-analysis}). Nevertheless, despite the
tremendous importance of this problem, very little is known about the
theoretical properties of many popular clustering techniques. Two such
fundamental properties are:
(i) stability under perturbation of the data; and
(ii) asymptotic consistency under specific data generating mechanisms
(i.e., convergence with probability $1$ as the sample size goes to
infinity under a fixed probability model).
Previous attempts to give theoretical justifications for both of these
properties have relied on a choice of a particular probability model.

In the present paper, we address these problems in a model-free way, by
recognizing both problems as following from the continuity of certain
maps on graphon space \cite{Lovasz:2013}.
To elaborate, clustering data can naturally be formulated as a graph
partitioning problem. Indeed, in applications, one generally uses a
\textit{similarity function} $f: \mathcal{X} \times \mathcal{X} \to \R$
to construct a similarity matrix $W = (w_{ij})$ where $w_{ij} = f(x_i,
x_j)$ measures the similarity between two data points $x_i, x_j$ in a
suitable space $\mathcal{X}$. The matrix $W$ is naturally identified with
a weighted graph. Another common approach is to use $W$ to build an
unweighted graph, where nodes are adjacent if and only if they are
similar enough (see e.g.~\cite[Section 2.2]{vonLuxburg-tutorial} for more
details). The problem of identifying clusters within data can thus be
reduced to partitioning the vertices of a graph, in such a way that nodes
belonging to the same cluster are well-connected together, whereas
different clusters share fewer edges.
In particular, all of the above algorithms are examples of graph
partitioning algorithms, that take a finite graph $G = (V(G), E(G))$ and
yield a partition of the vertices $V(G)$.
Formally, we can describe this as a map from finite graphs $(V(G),E(G))$
to \textit{$S$-colored graphs} $(V(G), E(G), c_G: V(G) \to S)$, where $S$
is a finite set.

Giving theoretical justification for a choice of graph partitioning
algorithm is a notoriously ill-defined problem, and a satisfactory
solution has been elusive to the broader mathematics community. In this
paper we propose that such a graph partitioning algorithm ought to
satisfy a form of \textit{model-free structural consistency}: if the
structures of the input graphs converge then the structures of the
partitioned graphs should also converge.
Such structural consistency subsumes both stability under perturbation as
well as asymptotic consistency as special cases. Formally, the
aforementioned map from graphs to $S$-colored graphs should be continuous
under the canonical dense graph limit topologies, developed over graphs
in \cite{BCL:2010, BCLSV:2008, BCLSV:2012, LS:2006}, and generalized by
us to $S$-colored graphs.
We characterize continuity for a broad class of graph partitioning
algorithms, and prove model-free structural consistency for popular graph
partitioning algorithms such as spectral clustering.

To explain how our approach relaxes the assumptions in previous work in
the area, consider the case of spectral clustering.  The asymptotic
consistency of spectral clustering has been studied in many papers
\cite{belkin2003, belkin2005, gine2006, hein2006, hein2005, lafon2004,
vonLuxburg2008, vonLuxburg2004}. As far as we are aware, all of these
results assume that the similarity graph $G_n$ is being generated
according to the following general procedure: pick $(x_i)_{i =1}^\infty$
i.i.d.~from some probability space $(\mathcal{X},\mu)$ and then compute
the similarity between node $i$ and $j$ to be $f(x_i,x_j)$ for some
function $f: \mathcal{X} \times \mathcal{X} \to \R$.
The aforementioned papers have worked to prove consistency of spectral
clustering for more and more general probability spaces $\mathcal{X}$ and
similarity functions $f$, usually exploiting some underlying geometric
structure, and to our knowledge with the most general results established
in \cite{vonLuxburg2008}.

Compared to previous work in the area, our approach provides a way to
establish the consistency of clustering algorithms without making
\textit{any} assumption about the exact form of the data generating
mechanism. Our only assumption is that data is provided in a
\textit{coherent} way. More precisely, we assume the graphs converge
structurally in the sense of the theory of dense graph limits. In
particular, when graphs are constructed from an i.i.d.~sequence $(x_i)_{i
=1}^\infty$ using a similarity function as above, it is known that the
resulting graphs converge almost surely to a limiting object (see Remark
\ref{GnW}). Therefore, the present paper extends previous results from
the literature.
We remark that the theory of graphons is indispensable to this paper for
two reasons:
(i) it provides a language to formulate model-free structural consistency
as a continuity result; and
(ii) its canonical topology ensures the broad applicability of the
framework, as we now explain.

Replacing the i.i.d.~assumption by the significantly weaker paradigm of
dense graph convergence, allows us to provide a novel statistical
framework to help handle two common problems in modern data analysis:
lack of a plausible data generating mechanism for complex data, and lack
of a mathematical representation space for inferred objects with no
linear structure. 
For these reasons, we believe our approach has an important advantage in
network analysis. Indeed, finding models that reflect the complex
heterogeneities of massive real-world networks still remains an important
challenge \cite{Leskovec2009}. The assumption that observations are
independent is also rarely verified in practice. In contrast, our
\textit{model-free} approach provides consistency results in a setting
that is broadly applicable.

The rest of the paper is structured as follows. In Section
\ref{Sinformal}, we provide an informal description of our main results.
We briefly review the theory of graph limits in Section \ref{Sreview},
and show in Section \ref{section:prelim} how the classical theory of
graph limits can naturally be extended to study the space of colored
graph sequences and their limit objects. In Section
\ref{section:nodelvl}, we study a common method of clustering the
vertices of a graph by computing some statistic for each vertex such as
its degree. We term these \textit{node-level statistics}, and prove a
general theorem about the structural consistency of such clustering
algorithms. Finally, in Section \ref{section:spectral}, we study the
structural consistency of spectral clustering in the graph limit
framework. We demonstrate that normalized spectral clustering is
structurally consistent under mild assumptions. We also demonstrate
problems with the analogous unnormalized procedure, as was previously
observed in \cite{vonLuxburg2008}. Proofs of technical results in Section
\ref{section:prelim} are provided in Appendix \ref{Acolor}, and in
Appendix \ref{AppB} we extend the Riesz--Fischer Theorem to any complete
metric space, as it is required to formulate one of our main results,
Theorem \ref{TclusterGen}.

\subsection{Informal statements of results}\label{Sinformal}

In this subsection, we explain informally the main results in the present
paper. The technical details and the results are discussed in full, in
later sections.

We begin by introducing the ingredients used to state and prove the main
results. The first notion is that of \textit{graphons}, which are
limiting objects of graph sequences. Graphons are measurable functions $W
: [0,1]^2 \to [0,1]$ that are symmetric, i.e., $W(x,y) = W(y,x)$. Every
graph is naturally identified with a graphon (see Equation
\eqref{Egraphon}). 

The space of graphons is equipped with a canonical topology. Suppose
$(G_n)_{n \geq 1}$ is a sequence of graphs. Let $t(K_2, G_n)$ denote the
\textit{edge density} of $G_n$, i.e., the proportion of pairs of vertices
of $G_n$ that are adjacent. More generally, given a simple graph $H$, we
denote by $t(H,G_n)$ the proportion of maps $H \to G_n$ that are edge
preserving. We say that a sequence of graphs $G_n$ is
\textit{left-convergent} if $t(H, G_n)$ is a convergent sequence of real
numbers for every simple graph $H$. The motivation behind
left-convergence comes from the notion that graphs become more and more
similar if their edge densities, triangle densities, etc., are all
convergent. A left-convergent sequence of graphons $(G_n)_{n \geq 1}$ is
naturally identified with a limit graphon. In order to do so, one first
extends the notion of homomorphism density to graphons (see Equation
\eqref{Ehom}), and then show that there exists a graphon $W \in
\W_{[0,1]}$ such that $t(H, G_n) \to t(H, W)$. The resulting topology is
metrizable by the \textit{cut-norm} (see Equation \eqref{Ecutnorm}). The
cut-norm provides a natural way of comparing graphs, even if they have
different numbers of vertices. Moreover, under the cut-norm, the space
$\W_{[0,1]}$ is a compact pseudo-metric space.

For more details, the reader is referred to Section \ref{Sreview}, and to
\cite{BCLSV:2008, Lovasz:2013} for a comprehensive introduction to the
theory of graphons.

In this paper, we introduce a new mathematical framework to study the
structural consistency of clustering algorithms. Given a graph $G$, we
identify a clustering of the vertices of $G$ with a coloring of $G$,
i.e., a map $c_G : V(G) \to S$ that assigns a ``color'' to every vertex
of $G$. We say that the clustering procedure is \textit{structurally
consistent} if for every left-convergent sequence of graphs $G_n$, the
resulting sequence of \textit{colored} graphs is also convergent (under
an appropriate topology similar to the canonical topology in the graphon
space). Note that in previous work in the literature, a clustering
procedure is consistent if the graphs $G_n$ with their colorings converge
whenever the graphs $G_n$ are generated i.i.d.~from a probability model.
By the theory of dense graph limits, one can show that each such sequence
$(G_n)_{n \geq 1}$ is convergent almost surely (see Remark \ref{GnW}).
Our approach thus significantly generalizes previous work by establishing
structural consistency in a ``model-free'' way, and without assuming
independence of the samples.

We now discuss our first main result, a very general clustering recipe.
Given a graph, there are several ways to cluster its nodes based on local
statistics. For instance, a simple clustering procedure involves
clustering nodes according to whether their degree (or edge-density) is
above or below a certain threshold value. More generally, one can work
with a finite collection of local statistics such as edge-densities,
triangle counts, and other graph homomorphism densities, where the images
of these graph morphisms involve the given node. Now the nodes are
clustered based on the tuple of values of such local statistics; note
that such tuples lie in Euclidean space.

In our first main result, we distill the essence of these clustering
recipes into the notion of a \textit{node-level statistic}. This is a
continuous map that sends a pair -- a graph(on) and a node on it -- to a
tuple in Euclidean space as above, or in full generality, to an arbitrary
metric space $X$. Our first main result establishes the structural
consistency of such general clustering procedures.\medskip

\noindent \textbf{Theorem \ref{TclusterGen}.} (See Section
\ref{section:nodelvl}.)
\textit{Fix a metric space $(X,d_X)$. Clustering according to any
continuous $X$-valued node-level statistic $f$ is structurally consistent
with respect to graph convergence. Namely, if a sequence of graphs is
convergent in the cut-norm, then clustering according to $f$ yields a
sequence of colored graphs that is also convergent.}\medskip

As a concrete example, Theorem \ref{TclusterGen} implies the structural
consistency of degree-based clustering as described above -- see Theorem
\ref{Tdeg-cluster} for a precise formulation. 

We remark here that partitioning a graph according to the degree
statistic was previously studied in the context of nonparametric graphon
estimation; see e.g.~\cite{BCCG}. In that work, the partitioning of nodes
is an intermediary step towards graphon estimation. In contrast, in the
present paper, we are chiefly concerned with the structural consistency
of the graph partitioning step itself. Furthermore, we do not take the
graphon as a nonparametric generating mechanism for graphs, but rather as
a general limit object for graphs in the graphon topology.

Note that Theorem \ref{TclusterGen} decouples the clustering recipe from
any graph generating mechanism, i.i.d.~or not, and assumes only that the
graph sequence converges in a canonical topology. Thus, Theorem
\ref{TclusterGen} provides a very general and broadly applicable recipe
for clustering. 

As special cases of Theorem \ref{TclusterGen}, we mention two algorithms
studied in the paper: (a) the aforementioned instances of clustering
according to tuples of homomorphism densities (see Section
\ref{Shomdensities}); and (b) spectral clustering according to the
normalized Laplacian (see Section \ref{section:spectral}). We believe the
result should also be broadly applicable to other popular clustering
algorithms, with minimal assumptions on the graph generation
process.\smallskip

The remaining two main results of the paper involve the structural
consistency of spectral clustering. The spectral clustering procedure
involves working with the \textit{normalized Laplacian} of a graph, and
more generally, of a graphon. Our second main result demonstrates that
the normalized Laplacians of a convergent sequence of graphons are also
convergent.\medskip

\noindent \textbf{Theorem \ref{Tconv_norm_lap}.} (See Section
\ref{SnormLap}.)
\textit{Suppose $W_n$ is a sequence of graphons that converges in
cut-norm to a graphon $W_0$, whose degree function $d_0(x)$ is positive
for almost every $x \in [0,1]$. Let $L'_{W_n}$ denote the corresponding
normalized Laplacian for $n \geq 0$. Then $L'_{W_n}$ converges to
$L'_{W_0}$.}\medskip

Theorem \ref{Tconv_norm_lap} extends the corresponding result in
\cite{vonLuxburg2008} without the assumption that the degree functions
are bounded below by a positive constant, and without the assumption that
the graphs $G_n$ are generated by an i.i.d.~mechanism.

It may be wondered if the assumption that $d_0(x) > 0$ a.e.~$x$, is
itself required in Theorem \ref{Tconv_norm_lap}. In Example
\ref{EnormLap}, we will show that this is indeed the case in order to
obtain a reasonable theory of consistency of spectral clustering.

Finally, we turn to our last main result, which proves structural
consistency of normalized spectral clustering in the model-free setting
of graph limits.\medskip

\noindent \textbf{Theorem \ref{TSpClusteringNorm}.} (See Section
\ref{Sconsistency}.)
\textit{Under appropriate assumptions, if $W_n$ is a convergent sequence
of graphons, and $(W_n, c_n)$ is a coloring of $W_n$ obtained via
normalized spectral clustering, then $(W_n, c_n)$ is also
convergent.}\medskip

Theorem \ref{TSpClusteringNorm} establishes the structural consistency of
the widely used normalized spectral clustering technique, without making
any assumptions on the data generating model. Note that this implies the
classical notion of statistical consistency for normalized spectral
clustering.

We believe the approach we provide in this paper can also be applied to
other clustering procedures. More generally, it is our hope that the
philosophy and framework developed in the paper will be used as an
inspiration to establish model-free consistency results for statistical
estimation and machine learning procedures coming from various areas.

\section{Review of dense graph limits}\label{Sreview}

We now briefly review dense graph limit theory \cite{BCLSV:2008,
Lovasz:2013}. This section serves to set notation, as well as to motivate
the next section on colored graph limit theory. The reader who is already
familiar with the theory of dense graph limits can safely skip this
section.

A \textit{graphon} is a bounded symmetric measurable function $f: [0,1]^2
\to [0,1]$. Each finite simple labelled graph $G$ with vertex set $V(G) =
\{ 1, 2, \dots, n \}$ can be naturally identified with the following
graphon $f^G$:
\begin{equation}\label{Egraphon}
f^G(x,y) := {\bf 1}_{(\lceil nx \rceil, \lceil ny \rceil) \in E(G)} =
\begin{cases}
    1, 	&\text{if $(\lceil nx \rceil, \lceil ny \rceil)$ is an edge in
    $G$,}\\
    0,	&\text{otherwise.}
\end{cases}
\end{equation}

The topology on isomorphism classes of finite simple graphs can be
described as follows. The graphon space $\W_{[0,1]}$ sits inside
$\mathcal{W}$, the vector space of bounded symmetric measurable functions
$f: [0,1]^2 \to \R$. Recall that $\W$ is equipped with a seminorm called
the \textit{cut-norm}
\begin{equation}\label{Ecutnorm}
\|f\|_\Box := \sup_{A,B \subset [0,1]} \left| \int_{A \times B} f(x,y)\
dx\ dy \right|
\end{equation}

\noindent where the supremum is taken over all Lebesgue measurable
subsets $A,B \subset [0,1]$. The group of measure-preserving bijections
$S_{[0,1]}$ acts on $\W_{[0,1]}$ as follows: given $\sigma \in S_{[0,1]}$
and $f \in \W_{[0,1]}$, define $f^\sigma(x,y) := f(\sigma(x),\sigma(y))$.
Now define
\begin{equation*}
\delta_\Box(f,g) := \inf_{\psi \in S_{[0,1]}} \boxnorm{f - g^\psi}.
\end{equation*}

Observe that $\delta_\Box(f^G, f^{G'}) = 0$ whenever $G,G'$ are
isomorphic finite simple graphs, so $\delta_\Box$ metrizes convergence of
isomorphism classes of finite graphs (up to blowups).

The topology induced by $\delta_\Box$ can also be described using
homomorphism densities.  Given graphs $G = (V(G), E(G))$, $H = (V(H),
E(H))$, denote by $\Hom(H,G)$ the set of edge-preserving maps from
$V(H)$ into $V(G)$, and define the homomorphism densities as follows:
\begin{equation}\label{Ehom-density}
t(H,G) := \frac{|\Hom(H,G)|}{|V(G)|^{|V(H)|}}. 
\end{equation}

\noindent Now a sequence of finite simple graphs $(G_n)_{n=1}^\infty$ is
said to \textit{left-converge} if for all finite simple graphs $H$, the
sequence $t(H,G_n)$ converges as $n \to \infty$. Intuitively, a graph
sequence $(G_n)$ left-converges if the graphs $G_n$ become more and more
similar, in that their edge densities, triangle densities, and so on, all
converge.

Observe that the definition of homomorphism densities extends to
arbitrary graphons $f$ as follows:
\begin{equation}\label{Ehom}
t(H,f) := \int_{[0,1]^k} \prod_{(i,j) \in E(H)} f(x_i, x_j)\ d x_1 \cdots
d x_k, \qquad k = |V(H)|.
\end{equation}

\noindent This is compatible with the graph statistics $t(H,G)$, in that
$t(H,f^G) = t(H,G)$ for all finite simple graphs $H,G$. An important
result in the theory of graphons is that $\delta_\Box$ metrizes
left-convergence \cite[Theorem 3.8]{BCLSV:2008}. More precisely, upon
identifying graphons $W \sim W'$ whenever $\delta_{\Box}(W,W') = 0$, the
space $\W_{[0,1]} / \sim$ of equivalence classes of graphons is a metric
space. The following result explains how graphons are limiting objects
for left-convergent dense graph sequences.

\begin{theorem}[Borgs--Chayes--Lov\'asz--S\'os--Vesztergombi,
\cite{BCLSV:2008}]\label{thm:BCLSV2008}
Let $(W_n)_{n=1}^\infty \subset \W_{[0,1]}$ be a sequence of graphons.
Then the following are equivalent:
\begin{enumerate}
\item $t(H,W_n)$ converges for all finite simple graphs $H$, i.e., $W_n$
is left-convergent.
\item $W_n$ is a Cauchy sequence in the $\delta_\Box$ metric.
\item There exists $W \in \W_{[0,1]}$ such that $t(H,W_n) \to t(H,W)$ for
all finite simple graphs $H$.
\end{enumerate}
Furthermore, $t(H,W_n) \to t(H,W)$ for all finite simple graphs $H$ for
some $W \in \W_{[0,1]}$ if and only if $\delta_\Box(W_n,W) \to 0$.
\end{theorem}

\noindent \textbf{Sampling.}
In addition to the cut metric and left-convergence, a third, equivalent
way to think of graph convergence is via sampling. Given a graphon $W$,
let $\mathbb{H}(n,W)$ denote a random weighted graph generated by
sampling i.i.d.~variables $(X_i)_{i =1}^n$ uniformly on $[0,1]$, and then
setting $W(X_i,X_j)$ to be the weight between nodes $i$ and $j$. Given a
weighted graph $H$ with $n$ vertices, let $\mathbb{G}(H)$ denote the
graph $G$ on $n$ vertices where for $i > j$, $(i,j) \in E(G)$ with
probability $H(i,j)$ and $G$ is made symmetric. Now let $\mathbb{G}(n,W)
:= \mathbb{G}(\mathbb{H}(n,W))$.
Then the probabilities $\bp(\mathbb{G}(n,W) = H)$ can be computed from
the homomorphism densities $t(H,W)$ by inclusion-exclusion formulas
\cite[Section 5.2.3]{Lovasz:2013}. Therefore, the left-convergence of a
graphon sequence $W_n$ is equivalent to convergence of the sampling
densities $\mathbb{G}(k, W_n)$ for all $k$.

\begin{remark}\label{GnW}
The sampling distributions $\mathbb{H}(n,W)$ and $\mathbb{G}(n,W)$ are
used as nonparametric generative models for networks. Here we remark that
both models $\mathbb{G}(n,W)$ and $\mathbb{H}(n,W)$ concentrate around
$W$ in the cut-distance (see \cite[Lemma 10.16]{Lovasz:2013}). In
particular, $\mathbb{G}(n,W)$ and $\mathbb{H}(n,W)$ converge almost
surely to $W$.
Note that the general data generating mechanism employed in
\cite{vonLuxburg2008} falls within the class $\{ \mathbb{H}(n,W) : W \in
\W_{[0,1]} \}$, because an arbitrary probability distribution on a
compact metric measure space $\mathcal{X}$ can always be mapped in a
measure-preserving fashion onto $[0,1]$ with the Lebesgue measure. Under
this mapping, the continuous (symmetric) similarity function $k :
\mathcal{X} \times \mathcal{X} \to [0,\infty)$ defines an associated
kernel $W_k \in \W$, bounded above uniformly by $m_k := \max_{x,y \in
\mathcal{X}} k(x,y) > 0$. Therefore, graphs $G_n$ generated i.i.d.~in
this model satisfy:
\[
m_k^{-1} f^{G_n} \sim \mathbb{H}(n, m_k^{-1} W_k),
\]

\noindent whence the random weighted graph sequence $m_k^{-1} G_n$
converges almost surely to the graphon $m_k^{-1} W_k \in \W_{[0,1]}$ as
$n \to \infty$. Thus the framework in the present paper applies to the
general setting of \cite{vonLuxburg2008}.
\end{remark}

We conclude with two additional facts about the metric space
$(\W_{[0,1]}/\sim, \delta_\Box)$:
\begin{enumerate}
\item (See \cite{LS:2006}.) The countable set of graphons $f^G$ (running
over all finite simple graphs $G$) is dense in $(\W_{[0,1]}/\sim,
\delta_\Box)$.

\item As a consequence of the Weak Regularity Lemma in graph theory,
Lov\'asz and Szegedy showed in \cite{LS:2007} that $(\W_{[0,1]}/\sim,
\delta_\Box)$ is a compact metric space.
\end{enumerate}  

\section{Colored graph limit theory}\label{section:prelim}

We begin by showing how the theory of dense graph limits can be extended
to colored graph sequences and their limits. The proofs of the results
stated in this section are given in Appendix \ref{Acolor}, and leverage
the approach of \cite{Lovasz:2013}.

\subsection{Colored graphs}

Let $S$ be a finite set. Define an \textit{$S$-colored graph} $G$ to be a
triple
\[
(V(G), E(G), c_G : V(G) \to S),
\]
where $V(G)$ and $E(G) \subset V(G)^2$ are finite sets. (We assume that
$G$ does not have multiple edges or self-loops.) Now let $\G_S$ denote
the set of \textit{$S$-colored graphs}.

Given $G \in \G_S$ and $s \in S$, we let $V_s(G) := \{v \in V(G) : c_G(v)
= s\}$ denote the set of vertices of $G$ of color $s$. For $H, G \in
\G_S$, we define the \textit{colored homomorphism density} by: 
\[
t_S(H,G) := \frac{|\Hom_S(H,G)|}{|V(G)|^{|V(H)|}}, 
\]

\noindent where $\Hom_S(H,G)$ denotes the set of edge preserving maps
$\phi: V(H) \to V(G)$ such that $c_H = c_G \circ \phi$.\medskip

For instance, if $H$ denotes the graph with one vertex, colored $s$, then
$t_S(H,G)$ precisely equals $|V_s(G)| / |V(G)|$.

Note that the colored homomorphism densities naturally generalize the
usual homomorphism densities (see Equation \eqref{Ehom-density}) in the
case where the graphs are uncolored.

\subsection{Colored graphons and homomorphism densities}

We now come to the limiting objects of sequences of colored graphs. We
define an \textit{$S$-colored graphon} to be a pair of measurable maps
$(f_W,c_W)$ where $f_W: [0,1]^2 \to [0,1]$ is symmetric and $c_W: [0,1]
\to S$. We denote the set of $S$-colored graphons by $\W_S$. Given $H \in
\G_S$ and $W \in \W_S$, we let 
\begin{equation}
t_S(H,W) := \int_{[0,1]^{V(H)}} \prod_{e \in E(H)} f_W(x_{e_s}, x_{e_t})
\prod_{v \in V(H)} {\bf 1}_{c_W(x_v) = c_H(v)} \prod_{v \in V(H)}dx_v. 
\end{equation}

\noindent In other words, the integration is carried out only over the
sub-rectangle given by:
\[
x_v \in c_W^{-1}(c_H(v)), \qquad \forall v \in V(H).
\]

The space of $S$-colored graphs $\G_S$ naturally embeds into $\W_S$ in
the following way. Let $k := |S|$, and enumerate $S = \{ s_1, \dots, s_k
\}$ in some fixed order. Given $G \in \G_S$ and $j \in \{1,\dots,k\}$,
let $p_0 := 0$, $p_j := |V_{s_j}(G)|/|V(G)|$, and let $I_j$ denote the
interval $I_j := (\sum_{l=0}^{j-1} p_l, \sum_{l=0}^j p_l]$. Now define
the $S$-colored graphon $G \leadsto W_G \in \W_S$, via:
\begin{align*}
&c_{W_G}(I_j) := j, \ c_{W_G}(1) = k, \\
&f_{W_G}(I_j \times I_{j'}) := {\bf 1}_{(j,j') \in E(G)},
\end{align*}
\noindent and $f_{W_G} = 0$ otherwise. The graphon $W_G$ is related to
the original graph $G$ as follows.

\begin{lemma}
For all $S$-colored graphs $H, G \in \G_S$, we have $t_S(H,G) =
t_S(H,W_G)$.
\end{lemma}

Recall that in the uncolored case, homomorphism densities are used to
construct a topology on the space of graphons (see Equation
\eqref{Ehom-density} and the subsequent paragraph). In a similar way, we
use colored homomorphisms to construct a topology on the space of colored
graphons.

\begin{definition}\label{Dleft}
A sequence of $S$-colored graphons $W_n \in \W_S$ is said to
\textit{left-converge} (to a graphon $W \in \W_S$) if the corresponding
sequence of colored homomorphism densities $t_S(H, W_n)$ converges (to
$t_S(H,W)$) for every fixed $S$-colored graph $H$.
\end{definition}

Note that when the nodes of a sequence of graphs all have the same color,
the above notion of left-convergence reduces to the usual uncolored
notion of left-convergence.

\subsection{Cut metric}

Recall that in the uncolored dense limit theory, the topology induced by
homomorphism densities can be metrized using the cut-norm (see Equation
\eqref{Ecutnorm}). We now extend the definition of the cut-norm to
$\W_S$:
\begin{equation}\label{ECutColor}
\boxnorm{W - W'}^S := \boxnorm{W - W'} + \sum_{s \in S} \mu_L(c_W^{-1}(s)
\Delta c_{W'}^{-1}(s)),
\end{equation}

\noindent where $\mu_L$ denotes the usual Lebesgue measure. Notice
$\boxnorm{\cdot}^S$ is not an actual norm when $|S|>1$; however, we
retain the present notation to maintain consistency with the uncolored
case $|S|=1$. Using this definition, we can naturally extend the usual
Counting Lemma to $\W_S$.

\begin{lemma}[Counting Lemma]\label{LcountingColor}
Let $H \in \G_S$ and $W, W' \in \W_S$. Then
\[
|t_S(H,W) - t_S(H,W')| \leq |E(H)| \cdot \boxnorm{W-W'} + \sum_{s \in S}
\mu_L(c_W^{-1}(s) \Delta c_{W'}^{-1}(s)).
\]
In particular, $|t_S(H,W) - t_S(H,W')| \leq |E(H)| \cdot
\boxnorm{W-W'}^S$.
\end{lemma}

Measure preserving maps $\sigma \in S_{[0,1]}$ naturally act on $\W_S$:
if $(W, c_W) \in \W_S$, then we let 
\begin{equation}\label{EsigmaColor}
W^\sigma(x,y) := W(\sigma(x), \sigma(y)) \qquad \textrm{and} \qquad
c_W^\sigma(x) := c_W(\sigma(x)).
\end{equation}

\noindent As in the uncolored case, we define the distance
$\delta_{\Box}^S$ for $W_1,W_2, \in \W_S$ by the formula
\[
\delta_{\Box}^S(W_1,W_2) := \inf_{\sigma \in S_{[0,1]}} \boxnorm{W_1 -
W_2^\sigma}^S,
\]

\noindent As usual, we will say $W_1 \sim W_2$ if
$\delta_{\Box}^S(W_1,W_2) = 0$.  Observe that $(\W_S/\sim,
\delta_\Box^S)$ is a metric space. It is in fact compact, as in the
classical case where the vertices are not colored.

\begin{theorem}\label{Tcompact}
The space $(\W_S/\sim, \delta_\Box^S)$ is compact.
\end{theorem}

Moreover, the colored cut-distance provides a way to metrize the topology
induced by the colored homomorphism densities. 

\begin{theorem}\label{thm:Top_Equiv}
Let $W_n$ be a sequence of $S$-colored graphons. Then the sequence $W_n$
left-converges (see Definition \ref{Dleft}) if and only if it is Cauchy
in the $\delta_\Box^S$ metric.
\end{theorem}

Finally, the colored finite graphs are dense in the completion $\W_S /
\sim$.

\begin{theorem}\label{thm:density}
Colored graphs are dense in $\W_S / \sim$.
\end{theorem}

Using an inclusion-exclusion argument, it is also not hard to show that
convergence in this topology is equivalent to convergence in a sampling
topology for $S$-colored graphs. Therefore, this topology is the
canonical topology for convergence of dense $S$-colored graphs.

\section{Clustering by continuous node-level
statistics}\label{section:nodelvl}

One general class of graph partitioning algorithms proceed as follows:
compute a statistic defined on the set of nodes and then partition the
nodes based on the value of that statistic.  For example, this includes
clustering based on the degree, local clustering coefficient, and
spectral clustering.  In this section, we call such statistics
\textit{node-level statistics}.  We introduce the notion of a
\textit{continuous node-level statistic}, which is broad enough to apply
to commonly used node-level statistics (such as the degree or local
clustering coefficient). We then show a general consistency result for
graph partitioning algorithms based on continuous node-level statistics.
We also verify continuity of node-level statistics defined based on local
graph statistics, thereby showing structural consistency of graph
partitioning based on such statistics.  The same general theorem is
applied in Section \ref{section:spectral} to show structural consistency
of spectral clustering, though such a result requires additional
machinery.

We begin by defining the notion of a node-level statistic and a
continuous node-level statistic.

\begin{definition}
Let $\G$ denote the set of labelled finite simple graphs. A
\textit{node-level statistic in $\R^m$} is a collection of maps $\{ f_G :
V(G) \to \R^m, G \in \G \}$.
\end{definition}

Notice that node-level statistics can be summarized also as a map $f :
\G \to L^1([0,1], \R^m)$, with the restriction that if we label
$V(G) := \{ 1, \dots, n \}$ then
\[
f(G)(y) = f_G(\lceil ny \rceil), \qquad G \in \G, y \in [0,1].
\]

Well-known examples of node-level statistics include the degree of a
node, or the local clustering coefficient, i.e., the proportion of pairs
of neighboring nodes that are adjacent.

\begin{definition}\label{DcontNLS}
We say that a node-level statistic in $\R^m$ is \textit{continuous} if
whenever $G_n \to W_0$ in the cut-norm, the family of functions $f(G_n)
\in L^1([0,1], \R^m)$ is convergent in $L^1$.
\end{definition}

Our notion of continuity above was defined so that natural node-level
statistics such as the degree would be continuous. For example, the
degree distribution does not necessarily converge pointwise as a sequence
of graphs converges in the cut-norm, but as we will show later, it
converges in $L^1$. Another candidate notion of convergence for functions
$f: [0,1] \to \R^m$ could be convergence with respect to a generalized
cut-norm
\[
\boxnorm{f} := \sup_{B \subset [0,1]} \left\| \int_B f(y) dy \right\|.
\]

\noindent However, as we now show, convergence in the generalized
cut-norm is equivalent to convergence in the $L^1$ norm, since the
function is defined on the interval as opposed to a higher dimensional
hypercube. To explain the equivalence, for every tuple ${\bf c} = (c_1,
\dots, c_m) \in \{0,1\}^m$, let $B_{\bf c} := f^{-1} ((-1)^{c_1} \R_{\ge
0} \times \cdots \times (-1)^{c_m} \R_{\ge 0})$. Then, 
\begin{align*}
\boxnorm{f} \le \|f\|_1 = \sum_{{\bf c} \in \{0,1\}^m} \left\|
\int_{B_{\bf c}} f(y) dy \right\| \le 2^m \boxnorm{f},
\end{align*}

\noindent so the cut-norm and the $L^1$ norm are equivalent.

Another advantage of using the notion of $L^1$ is that it generalizes
naturally to node-level statistics taking values in arbitrary metric
spaces $(X,d_X)$. In the following definition, $L^1([0,1],X)$ is the set
of measurable functions
\[
L^1([0,1],X) := \{ g : ([0,1], \mu_L) \to (X, \mathscr{B}_X) : 
\int_0^1 d_X(x_0, g(y))\ dy < \infty \}
\]
for any choice of point $x_0 \in X$, and equipped with the metric
\[
d_1(g,g') := \int_0^1 d_X(g(y), g'(y))\ dy.
\]

\noindent Here, $\mu_L$ stands for Lebesgue measure, and $\mathscr{B}_X$
for the Borel $\sigma$-algebra on $X$. As usual, we identify functions in
$L^1([0,1],X)$ that are equal almost everywhere on $[0,1]$.

\begin{definition}\label{Dnode-stat-metric}
Given a metric space $(X, d_X)$, a \textit{node-level statistic in $X$}
is $f : \G \to L^1([0,1],X)$, such that $f(G)(y) = f_G(\lceil ny \rceil)$
for all $G \in \G, y \in [0,1]$.
We say that an $X$-valued node-level statistic $f : \G \to L^1([0,1],X)$
is \textit{continuous} if whenever $G_n \to W_0$ in the cut-norm,
\[
\lim_{m, n \to \infty} d_1(f(G_n), f(G_m)) = 0,
\]

\noindent i.e., the sequence of functions $f(G_n) \in L^1([0,1],X)$ is
Cauchy.
\end{definition}

Note that this definition reduces to Definition \ref{DcontNLS} when $X =
\R^m$ (equipped with the usual Euclidean distance). In fact this is true
for any complete metric space $(X,d_X)$: continuous $X$-valued node-level
statistics $f : \G \to L^1([0,1],X)$ continuously extend to functions $f:
\W_{[0,1]} \to L^1([0,1],X)$. See Corollary \ref{Cmetric-ext} in Appendix
\ref{AppB} for details.

In the remainder of the paper, we will only deal with continuous
node-level statistics. Hence, from now on we take such a statistic to
denote a continuous function $f : \W_{[0,1]} \to L^1([0,1],X)$.
With a slight abuse in notation we shall also write $f(W,y)$ instead of
$f(W)(y)$.

Our next theorem provides a general consistency result when clustering is
performed using a continuous node-level statistic $f : \W_{[0,1]} \to
L^1([0,1], X)$. We use $f$ to color the vertices of $W \in \W_{[0,1]}$.
More precisely, suppose there exists a collection of disjoint open sets
$(A_j)_{j=1}^N \subset X$ such that
\[
f(W,y) \in \bigcup_{j=1}^N A_j \quad \textrm{for a.e. } y \in [0,1].
\] 

\noindent Then it is natural to define a coloring $c_W: [0,1] \to
\{1,\dots,N\}$ by letting $c_W(y)$ be the unique $j$ such that $f(W,y)
\in A_j$. Note that $c_W$ is well-defined for almost every $y \in [0,1]$.
Letting $S = \{1,\dots,N\}$, this operation induces a map $F: \W_{[0,1]}
\to \W_S$ defined by
\[
F(W) := (W, c_W).
\]

Structural consistency of the above coloring is equivalent to continuity
of the map $F$. Our next result provides a useful sufficient condition
for the map $F$ to be continuous. In what follows, we identify functions
$f: \W_{[0,1]} \to L^1([0,1],X)$ with functions $f: [0,1]^2 \times [0,1]
\to [0,1] \times X$.

\begin{utheorem}\label{TclusterGen}
Let $(X,d_X)$ be a metric space, and $f: \W_{[0,1]} \to L^1([0,1],X)$ a
continuous node-level statistic in $X$.
Given a collection of disjoint open sets $A_1, \dots A_N \subset X$,
define
\begin{equation}\label{Enomass}
\mathcal{D} := \{ W \in \W_{[0,1]} : f(W, y) \in \bigcup_{j=1}^N A_j
\quad \textrm{for a.e. } y \in [0,1] \}.
\end{equation}

\noindent Define $S = \{1, \dots, N\}$ and for $W \in \mathcal{D}$, let
$c_W: [0,1] \to S$ be the coloring defined for almost every $y$ by
letting $c_W(y) = j$, where $j$ is the unique color in $S$ such that
$f(W,y) \in A_j$. 
Then the map $F: (\mathcal{D}, \boxnorm{\cdot}) \to (\W_S ,
\boxnorm{\cdot}^S)$ given by $F(W) = (W, c_W)$ is continuous.
\end{utheorem}

Note, the theorem has a technical assumption about working with a subset
$\mathcal{D}$ rather than the full graphon space $\W_{[0,1]}$. In fact
this assumption is required, as we explain in Example \ref{Ex-deg-mass}.

We now show that under a supplementary invariance assumption, a similar
result holds on the quotient graphon spaces. 

\begin{definition}
Let $(X,d_X)$ be a metric space. We say that a continuous node-level
statistic $f : \W_{[0,1]} \to L^1([0,1],X)$ is $S_{[0,1]}$-invariant if 
\begin{equation}
f(W^\sigma, y) := f(W, \sigma(y)), \qquad \forall G \in \G, \ \sigma \in
S_{[0,1]}, \ a.e.\ y \in [0,1].
\end{equation}

\noindent In particular, for a graph $G$ with vertex set $\{ 1, \dots, n
\}$, and a permutation $\sigma \in S_n$, we have $f(G^\sigma, y) =
f(G,\sigma(y))$, where $\sigma$ acts on $[0,1]$ under the usual embedding
$S_n \hookrightarrow S_{[0,1]}$.\footnote{Specifically, each $\sigma \in
S_n$ acts on $[0,1]$ by fixing $1$, and sending $y \in ((i-1)/n, i/n]$ to
$y + (\sigma(i) - i)/n$.}
\end{definition}

\begin{cor}\label{CclusterGen}
In the setting of Theorem \ref{TclusterGen}, if in addition $\mathcal{D}$
is $S_{[0,1]}$-stable and $f$ is $S_{[0,1]}$-invariant so that
$f(W^\sigma,y) = f(W,\sigma(y))$ for all $\sigma \in S_{[0,1]}$ and
almost every $y \in [0,1]$, then the map $F: (\mathcal{D} / \sim,
\delta_\Box) \to (\W_S / \sim, \delta_\Box^S)$ given by $F(W) = (W, c_W)$
is continuous.
\end{cor}
 
We now prove Theorem \ref{TclusterGen}, and using it, Corollary
\ref{CclusterGen}.

\begin{proof}[Proof of Theorem \ref{TclusterGen}]
Suppose $W_n \to W_0$ as $n \to \infty$, with $W_n \in \mathcal{D}$ for
all $n \geq 0$.
For $j=1,\dots,N$ and $n \geq 0$, let 
\begin{align*}
E_{n,j} &:= \{ y \in [0,1] : f(W_n,y) \in A_j\}.
\end{align*}

\noindent For $k \geq 1$ and $j=1,\dots,N$, let 
\[
A_{j,k} := \left\{x \in A_j : \frac{1}{k+1} \leq d_X(x, A_j^c) <
\frac{1}{k}\right\}, 
\]
and let $A_{j,0} := \left\{x \in A_j : d_X(x, A_j^c) \geq 1\right\}$,
where we denote $d_X(x,A) := \inf \{d_X(x,a) : a \in A\}$. Note that $A_j
= \cup_{k=0}^\infty A_{j,k}$ since $A_j$ is open. For $k \geq 0$, define
\[
F_{j,k} := \left\{ y \in [0,1] : f(W_0,y) \in A_{j,k}\right\}. 
\]

\noindent By the definition of $A_{j,k}$ and assumption \ref{Enomass}, we
have  
\[
E_{0,j}^c = \bigcup_{\substack{l=1 \\ l \ne j}}^N
\bigcup_{k=0}^\infty F_{l,k} \cup Z
\] 

\noindent for some set $Z \subset [0,1]$ with $\mu_L(Z) = 0$. Moreover,
the sets $F_{l,k}$ are disjoint. Thus,  
\[
\mu_L(E_{n,j} \setminus E_{0,j}) = \sum_{\substack{l=1 \\ l \ne j}}^N
\sum_{k=0}^\infty \mu_L(E_{n,j} \cap F_{l,k}).
\]

\noindent Now fix $k \geq 0$ and distinct colors $j \neq l \in S$. For
every $n \geq 1$, we have by the definitions of $E_{n,j}$ and $F_{l,k}$
that 
\[
d_X \left(f(W_n,y), f(W_0,y)\right) \geq \frac{1}{k+1} \qquad \forall y
\in E_{n,j} \cap F_{l,k}.
\]

\noindent Therefore, by the continuity of the node-level statistic $f$,
\begin{align*}
\frac{1}{k+1} \cdot \mu_L(E_{n,j} \cap F_{l,k}) &\leq \int_{E_{n,j}
\cap F_{l,k}} d_X \left(f(W_n,y),f(W_0,y)\right) dy\\
&\leq \int_{[0,1]} d_X \left(f(W_n,y),f(W_0,y)\right)\ dy \to 0
\end{align*}

\noindent as $n \to \infty$. It follows that $\mu_L(E_{n,j} \cap F_{l,k})
\to 0$ as $n \to \infty$ for each fixed $k \geq 0$ and $l \ne j$.
Moreover,
\[
\mu_L(E_{n,j} \cap F_{l,k}) \leq \mu_L(F_{l,k}), \qquad \sum_{l \ne j}
\sum_{k=0}^\infty \mu_L(F_{l,k}) = \mu_L(E_{0,j}^c) < \infty.
\]

\noindent Therefore, by the dominated convergence theorem, 
\[
\mu_L(E_{n,j} \setminus E_{0,j}) = \sum_{l \ne j} \sum_{k=1}^\infty
\mu_L(E_{n,j} \cap F_{l,k}) \to 0
\]

\noindent as $n \to \infty$. By a similar argument, we also obtain
$\mu_L(E_{0,j} \setminus E_{n,j}) \to 0$ as $n \to \infty$. Therefore,
$\mu_L(E_{n,j} \Delta E_{0,j}) \to 0$ as $n \to \infty$. Since
$\boxnorm{W_n-W_0} \to 0$ as $n \to \infty$, this implies
$\boxnorm{W_n-W_0}^S \to 0$. This concludes the proof of the theorem.
\end{proof}

\begin{proof}[Proof of Corollary \ref{CclusterGen}]
Suppose $\delta_\Box(W_n, W_0) \to 0$ as $n \to \infty$. Then there
exists $(\sigma_n)_{n \geq 1} \subset S_{[0,1]}$ such that
$\boxnorm{W_n^{\sigma_n} \to W_0} \to 0$ as $n \to \infty$. By Theorem
\ref{TclusterGen},
$\boxnorm{(W_n^{\sigma_n}, c_{W_n^{\sigma_n}}) - (W_0, c_{W_0})}^S$ $\to 0$.
Now observe that by the invariance assumption on $f$, we have
\begin{equation*}
c_{W_n^{\sigma_n}}(y) = j \Leftrightarrow f(W_n^{\sigma_n}, y) \in A_j
\Leftrightarrow f(W_n, \sigma_n(y)) \in A_j \Leftrightarrow
c_{W_n}(\sigma_n(y)) = j.
\end{equation*}

\noindent It follows that the coloring defined by $f$ is consistent with
the $S_{[0,1]}$ action on $\W_S$ (see \eqref{EsigmaColor}). We therefore
conclude that $\delta_\Box^S((W_n,c_{W_n}),(W_0,c_{W_0})) \to 0$ as $n
\to \infty$.
\end{proof}

\subsection{Structural consistency of clustering by homomorphism
densities}\label{Shomdensities}

Our next goal is to provide a broad family of functions that satisfy the
hypotheses of Theorem \ref{TclusterGen}. Let $H$ be a $k$-labelled graph
with labelled vertices $1,\dots, k$. Recall that for $W \in \W$ and $x_1,
\dots, x_k \in [0,1]$, we define
\[
t_{x_1, \dots, x_k}(H,W) = \int_{[0,1]^{V(H) \setminus \{1,\dots,k\}} }
\prod_{e \in E(H)} W(x_{e_s}, x_{e_t})\ \prod_{v \not\in \{1,\dots,k\}}
dx_v.
\]
In particular, if $K_2$ has one labelled vertex, then 
\[
t_x(K_2, W) = \int_0^1 W(x,y)\ dy, \qquad x \in [0,1].
\]

\noindent The following result is a consequence of Lemma
\ref{LcountingGen} and Theorem \ref{TclusterGen}.

\begin{theorem}\label{Tdeg-cluster}
Fix $0 < \alpha < 1$ and a $1$-labelled graph $H$, and define
\begin{equation}\label{Enomass_gen}
\mathcal{D} = \mathcal{D}_{H,\alpha} := \{ W \in \W_{[0,1]} :
\mu_L\left(\left\{ y \in [0,1] : t_y(H, W) = \alpha\right\}\right) = 0
\}.
\end{equation}

\noindent Let $(W_n)_{n \geq 1} \subset \mathcal{D}$ such that $W_n \to
W_0 \in \W_{[0,1]}$, with $W_0 \in \mathcal{D}$. Let $S = \{1,2\}$ and
let $c_n : [0,1] \to S$ be defined for $n \geq 0$ by 
\[
c_n(y) := \begin{cases}
1, & \textrm{ if } t_y(H,W_n) < \alpha, \\
2, & \textrm{ if } t_y(H, W_n) \geq \alpha.
\end{cases}
\]

\noindent Then the sequence $(W_n, c_n)$ converges to $(W_0, c_0)$ in
$\mathcal{W}_S$.
\end{theorem}

In order to show the result, we first prove a generalization of the usual
Counting Lemma (see \cite[Lemma 10.24]{Lovasz:2013}). 

\begin{lemma}\label{LcountingGen}
Fix a finite simple graph $H$, graphons $W_e, W'_e \in \W_{[0,1]}$ for
all edges $e \in E(H)$, and measurable subsets $F_v \subset [0,1]$ for
all vertices $v \in V(H)$. Now define
\[
{\bf W} := (W_e)_{e \in E(H)}, \quad {\bf W'} := (W'_e)_{e \in E(H)},
\quad {\bf F} := \times_{v \in V(H)} F_v \subset [0,1]^{|V(H)|},
\]

\noindent as well as the following ``generalized homomorphism density'':
\begin{equation}
t_{{\bf F}}(H,{\bf W}) := \int_{{\bf F}} \prod_{e \in E(H)} W_e(x_{e_s},
x_{e_t}) \cdot \prod_{v \in V(H)} dx_v.
\end{equation}

\noindent Then,
\begin{equation}\label{Egen-count}
\left| t_{{\bf F}}(H,{\bf W}) - t_{{\bf F}}(H,{\bf W'}) \right| \leq
\sum_{e \in E(H)} \left(\prod_{v \neq e_s, e_t} \mu_L(F_v)\right)
\min\left\{ \mu_L(F_{e_s}) \mu_L(F_{e_t}), \boxnorm{ W_e - W'_e }
\right\}.
\end{equation}
\end{lemma}

\begin{proof}
We adapt the proof of \cite[Lemma 10.24]{Lovasz:2013}. We first claim
that for every edge $(u,v) \in E(H)$,
\begin{align*}
\left| \int_{{\bf F}} \prod_{\substack{(i,j) \in E(H) \\ (i,j) \ne
(u,v)}} W_{ij}(x_i, x_j) (W_{uv}(x_u, x_v)-W'_{uv}(x_u, x_v)) \ dx\right|
\leq& \left(\prod_{w \neq u,v} \mu_L(F_w)\right) \times \\
&\min( \mu_L(F_u) \mu_L(F_v), \boxnorm{ W_e - W'_e }).
\end{align*}

Indeed, note that the left-hand side of the expression can be written as 
\[
\left| \int_{{\bf F}} f(x) g(x) (W_{uv}(x_u, x_v)-W'_{uv}(x_u, x_v)) \
dx\right|, 
\] 
where 
\[
f(x) := \prod_{(i,j) \in \nabla(u) \setminus (u,v)} W_{ij}(x_i, x_j),
\qquad g(x) := \prod_{(i,j) \in E(H) \setminus \nabla(u)} W_{ij}(x_i,
x_j).
\]

\noindent Here $\nabla(u)$ denotes the set of edges with one endpoint
equal to $u$. Note that $f(x)$ does not depend on $x_v$ and $g(x)$ does
not depend on $x_u$. Thus, by \cite[Lemma 8.10]{Lovasz:2013}, 
\[
\left| \int_{F_u \times F_v} f(x) g(x) (W_{uv}(x_u, x_v)-W'_{uv}(x_u,
x_v)) \ dx_u dx_v\right| \leq \min( \mu_L(F_u) \mu_L(F_v), \boxnorm{W_e -
W'_e}).
\]

\noindent The claim follows by integrating with respect to the remaining
variables. We immediately obtain the desired result by writing $\left|
t_{{\bf F}}(H,{\bf W}) - t_{{\bf F}}(H,{\bf W'}) \right|$ as a
telescoping sum where each term is as in the claim.
\end{proof}

We now show that clustering according to homomorphism densities is
structurally consistent.

\begin{proof}[Proof of Theorem \ref{Tdeg-cluster}]
Let $f(W,y) := t_y(H, W)$. By Lemma \ref{LcountingGen}, for every
sequence $W_n \to W_0$ and every measurable subset $B \subset [0,1]$, 
\[
\left|\int_B [f(W_n,y) - f(W_0,y)]\ dy\right| \leq |E(H)| \cdot
\boxnorm{W_n - W_0}.
\]

\noindent Set $B_\pm := \{ y \in [0,1] : \pm (f(W_n,y) - f(W_0,y)) > 0
\}$. Then by the discussion after Definition \ref{DcontNLS},
\begin{align*}
\frac{1}{2} \| f(W_n) - f(W_0)\|_1 = &\ \frac{1}{2} \left| \int_{B_+}
(f(W_n,y) - f(W_0,y))\ dy \right| + \frac{1}{2} \left| \int_{B_-}
(f(W_n,y) - f(W_0,y))\ dy \right|\\
\leq &\ |E(H)| \cdot \boxnorm{W_n - W_0}.
\end{align*}

\noindent It follows that $f$ is a continuous node-level statistic $:
\W_{[0,1]} \to L^1([0,1],[0,1])$.
The result now follows by Theorem \ref{TclusterGen}, with $A_1 =
[0,\alpha)$ and $A_2 = (\alpha,1]$.
\end{proof}

\begin{remark}
Theorem \ref{Tdeg-cluster} easily extends to any finite set $H_1, \dots,
H_k$ of $1$-labelled graphs, and any collection of disjoint open sets
$A_1, \dots, A_N$ in the cube $[0,1]^k$. Namely, define
\begin{equation}
\mathcal{D} := \{ W \in \W_{[0,1]} : (t_y(H_i, W))_{i=1}^k \in
\bigcup_{j=1}^N A_j \textrm{ for a.e. } y \in [0,1] \}.
\end{equation}

\noindent Given $W \in \mathcal{D}$, define $c_W : [0,1] \to S := \{ 1,
\dots, N \}$ via: $c_W(y) = j$ if $(t_y(H_i,W))_{i=1}^k \in A_j$. Then
the clustering map $W \mapsto (W,c_W)$ is continuous on $\mathcal{D}$.
\end{remark}

\begin{ex}\label{Ex-deg-mass}
We briefly explain why the assumption $W_n \in \mathcal{D}$ is required
in Theorem \ref{TclusterGen} (instead of allowing all of $\W_{[0,1]}$).
Consider the case of edge-density, where $H = K_2$ consists of an edge.
Suppose $Y \subset [0,1]$ has positive measure, and $W_0 \in \W_{[0,1]}$
is such that
\[
W_0(x,y) \in (0,1)\ \forall x,y \in [0,1], \qquad
t_y(H,W_0) = \alpha\ \forall y \in Y.
\]

\noindent Fix any partition $Y = Y' \sqcup Y''$ into sets of positive
measure, and $\epsilon \in [0,1]$, define $W_{Y',\epsilon} \in
\W_{[0,1]}$ via:
\[
W_{Y',\epsilon}(x,y) = \begin{cases}
\epsilon + (1-\epsilon) W_0(x,y), \qquad & \text{if } x,y \in Y';\\
(1-\epsilon) W_0(x,y), & \text{if } x,y \in Y'';\\
W_0(x,y), & \text{otherwise.}
\end{cases}
\]

\noindent For each $Y'$, note that as $\epsilon \to 0^+$, we have
$W_{Y',\epsilon} \to W_0$ in $L^1$, hence in cut-norm. 
On the other hand, it is easily verified that for the graphon
$W_{Y',\epsilon}$,
\begin{alignat*}{5}
& \deg(y) = \alpha + \epsilon \int_{Y'} (1-W_0(x,y))\ dx && >\
\alpha, \qquad && \forall y \in Y',\\
& \deg(y) = \alpha - \epsilon \int_{Y''} W_0(x,y)\ dx && <\
\alpha, \qquad && \forall y \in Y''.
\end{alignat*}

\noindent Therefore different choices of $Y'$ would give inconsistent
limit clusters.
\end{ex}

\subsection{Sensitivity of clustering based on node-level statistics}

We conclude this section by discussing the sensitivity of the clustering
in Theorem \ref{TclusterGen}. A more sensitive notion of clustering would
be obtained if one could show that the function $F : W \mapsto (W,c_W)$
is Lipschitz on $\mathcal{D}$, as doing so would yield a greater
understanding of approximation errors. However, the following simple
example shows this is not always true.

\begin{ex}\label{Eerdosrenyi}
Suppose $H$ is any $1$-labelled graph with at least one edge, and
consider the clustering procedure in Theorem \ref{Tdeg-cluster}, for some
$\alpha \in (0,1)$. Now one can choose Erd\"os--R\'enyi graphs $W_1
\equiv p_1, W_2 \equiv p_2$ where $p_1^{|E(H)|} < \alpha < p_2^{|E(H)|}$.
Then the clustering algorithm shows that all vertices of $W_1$ are
colored $0$, while all vertices of $W_2$ are colored $1$, whence
$\delta_\Box^S(W_1,W_2) = 1$. On the other hand, $p_1,p_2$ can be chosen
arbitrarily close to one another, whence $\delta_\Box(W_1, W_2)$ can be
made as small as desired. It follows that $W \mapsto (W,c_W)$ is not
Lipschitz on $\mathcal{D}_{H,\alpha}$ (see Equation \eqref{Enomass_gen}).
\end{ex}

We now show that the problem in the previous example lies in the fact
that the sets $A_j$ are not necessarily separated. If instead they were
separated, in some sense ``discretizing'' the situation, then the
clustering map $F$ is Lipschitz, as long as the node-level statistic is.
More precisely, we have: 

\begin{theorem}\label{Tlipschitz}
Fix $d_{\min} > 0$, and suppose $A_1, \dots, A_N$ are disjoint open sets
in a metric space $(X, d_X)$, such that distinct sets $A_j$ are at least
$d_{\min}$ distance apart.
Also fix a continuous node-level statistic $f : \W_{[0,1]} \to
L^1([0,1],X)$, and define $\mathcal{D} \subset \W_{[0,1]}$ and the
clustering map $F : (\mathcal{D}, \boxnorm{\cdot}) \to (\W_S,
\boxnorm{\cdot}^S)$ as in Equation \eqref{Enomass}.

Now if the node-level statistic $f$ is Lipschitz, then so is the
clustering map $F$. More generally, suppose $\varphi : [0,\infty) \to
[0,\infty)$ satisfies
\[
d_1(f(W), f(W')) \leq \varphi ( \boxnorm{W - W'}), \qquad \forall W,W'
\in \mathcal{D}.
\]

\noindent Then $\boxnorm{F(W) - F(W')}^S \leq \varphi_1( \boxnorm{W -
W'})$, where $\displaystyle \varphi_1(y) := y +
\frac{\varphi(y)}{d_{\min}}$.
\end{theorem}

For certain continuous (even Lipschitz) node-level statistics $f$,
Theorem \ref{Tlipschitz} also has a converse: the clustering map $F$ is
Lipschitz if and only if the sets $A_j$ are separated. This is the case,
for example, for any $1$-labelled homomorphism density, for which this
converse was shown in Example \ref{Eerdosrenyi}.

\begin{proof}
Define $g : [0,1] \to [0,\infty)$ via: $g(y) := d_X(f(W,y), f(W',y))$.
Note that if $f(W,y), f(W',y)$ belong to distinct sets $A_j$, then their
distance is at least $d_{\min}$. Hence, we compute using Markov's
inequality:
\begin{align*}
\boxnorm{F(W) - F(W')}^S = &\ \boxnorm{W-W'} + \sum_{s \in S}
\mu_L(c_W^{-1}(s) \Delta c_{W'}^{-1}(s))\\
\leq &\ \boxnorm{W-W'} + \mu_L \{ y \in [0,1] : g(y) \geq d_{\min} \}\\
\leq &\ \boxnorm{W-W'} + \frac{1}{d_{\min}} \int_0^1 g(y)\ dy\\
= &\ \boxnorm{W-W'} + \frac{1}{d_{\min}} d_1(f(W),f(W'))\\
\leq &\ \boxnorm{W-W'} + \frac{1}{d_{\min}} \varphi(\boxnorm{W-W'})
= \varphi_1(\boxnorm{W-W'}),
\end{align*}

\noindent as desired.
\end{proof}

\begin{remark}
The proof also demonstrates that if we merely know $f$ is continuous,
then so is $F$. To see why, simply stop the preceding calculation before
the final inequality, and take $W' \to W$. This provides a second, easier
proof of the continuity of $F$ (when $f$ is continuous), in the simpler
setting where the open sets $A_j$ are separated in $X$.
\end{remark}

\section{Spectral clustering}\label{section:spectral}

One of the most popular algorithms for graph partitioning is spectral
clustering (see e.g.~\cite{vonLuxburg-tutorial}).  The algorithm proceeds
by constructing the graph Laplacian $L_G$ from a graph $G$ by taking a
diagonal matrix $D_G$ of the degrees of $G$ and subtracting the adjacency
matrix of $G$ from it.  This is known as the \textit{unnormalized
Laplacian}. In practice, it is also common to work with a
\textit{normalized} version of the Laplacian $L'_G$, which is derived
from the unnormalized Laplacian by normalizing both the rows and columns
by $D_G^{-1/2}$, and tends to give better results (see
\cite{vonLuxburg-tutorial} for more details).

Note that the Laplacian $L_G$ always has a constant eigenvector of
eigenvalue $0$. Moreover, the matrix $L_G$ is positive semidefinite as it
is diagonally dominant. Barring multiplicity issues, the first $m$
associated eigenvectors of the smallest non-zero eigenvalues of $L_G$
define an embedding of $V(G) \to \mathbb{R}^m$ by their coordinates, and
spectral clustering partitions the vertices of $G$ by how they cluster in
$\mathbb{R}^m$.  The same procedure can be carried out for the
normalized Laplacian $L_G'$.  To prove structural consistency of spectral
clustering, we introduce the notion of the Laplacian of a graphon $W$. We
then determine when the normalized and unnormalized Laplacians $L'_W,
L_W$ are continuous constructions under the cut topology. Finally, we
evaluate how the convergence of the spectrum to that of the limit
operators leads to structural consistency. As we will show, in general,
normalized spectral clustering has better consistency properties than
unnormalized clustering. Our results thus confirm previous findings from
\cite{vonLuxburg2008}.

As explained in \cite[Section 7.5]{Lovasz:2013}, a graphon $W \in
\W_{[0,1]}$ can be thought of as a self-adjoint integral operator $T_W$,
where 
\begin{equation}
T_W(f) = \int_0^1 W(x,y) f(y)\ dy.
\end{equation}

\noindent When we think of $T_W$ as an operator $T_W: L^\infty([0,1]) \to
L^1([0,1])$, then the operator norm of $T_W$ is equivalent to the
cut-norm (\cite[Lemma 8.11]{Lovasz:2013}): 
\begin{equation}\label{Ecutinfty1}
\boxnorm{W} \leq \| T_W \|_{\infty \to 1} \leq 4 \boxnorm{W}.
\end{equation}

\noindent We also often see $T_W$ as an operator $T_W: L^2 \to L^2$. In
that case, the resulting operator is Hilbert--Schmidt.  In particular,
$T_W$ has a countable spectrum, and can only have $0$ as an accumulation
point. We can therefore define the eigenvalues of $W$ to be the
eigenvalues of the associated Hilbert--Schmidt operator $T_W: L^2 \to
L^2$. We shall denote these eigenvalues by $\lambda_1(W), \lambda_2(W)$,
etc., where 
\[
|\lambda_1(W)| \geq |\lambda_2(W)| \geq \cdots.
\] 

\noindent Since $W$ is symmetric, the operator is self-adjoint. Therefore
we can choose an orthonormal basis $f_i \in L^2$ of eigenfunctions
associated to $\lambda_i(W)$ with appropriate multiplicities so that 
\begin{equation}\label{Espectral}
W(x,y) = \sum_{i=1}^\infty \lambda_i(W) f_i(x) f_i(y),
\end{equation} 

\noindent where $\|f_i\|_2 = 1$, and where \eqref{Espectral} converges in
the $L^2$ sense.

\subsection{Convergence of eigenvalues and eigenvectors}

We now examine the behavior of the eigenvalues of a sequence of graphons.

\begin{definition}[{\cite[Definition 1.1]{Szegedy2011}}]
Let $W \in \W_{[0,1]}$ have spectral decomposition as in
\eqref{Espectral} with eigenvalues $\lambda_i = \lambda_i(W)$. Given
$\lambda \geq 0$, we define a \textit{cutoff graphon} $[W]_\lambda$ by 
\begin{equation*}
[W]_\lambda(x,y) := \sum_{\{i : |\lambda_i| > \lambda\}} \lambda_i f_i(x)
f_i(y).
\end{equation*}
\end{definition}

Notice that for any sequence of graphons such that $\boxnorm{W_n - W_0}
\to 0$ as $n \to \infty$, we have 
\[
W_0 = \lim_{n \to \infty} \lim_{k \to \infty} [W_n]_{\alpha_k}
\]
in $L^2$, whenever $\alpha_k \to 0$ as $k \to \infty$. The following
theorem characterizes the convergence of graphons in the cut-norm, in
terms of the  convergence of its cutoffs $[W_n]_\lambda$. 

\begin{theorem}[{see \cite[Proposition
1.1]{Szegedy2011}}]\label{Tszegedy1}
Let $\{W_n\}_{n \geq 1} \subset \W_{[0,1]}$ and let $W_0 \in \W_{[0,1]}$.
Then the following are equivalent: 
\begin{enumerate}
\item $\boxnorm{W_n-W_0} \to 0$ as $n \to \infty$.
\item There is a decreasing sequence $\{\alpha_k\}_{k \geq 1} \subset
(0,\infty)$ with $\lim_{k \to \infty} \alpha_k = 0$ such that
$\|[W_n]_{\alpha_k}-[W_0]_{\alpha_k}\|_2 \to 0$ as $n \to \infty$ for
every $j$.
\end{enumerate}
Furthermore, if $(2)$ holds, then 
\begin{equation*}
W_0 = \lim_{k \to \infty} \lim_{n \to \infty} [W_n]_{\alpha_k}
\end{equation*}
in the $L^2$ sense. 
\end{theorem}

Using the above result, we can now understand the behavior of the
eigenvalues and eigenvectors of a sequence of graphons.

\begin{theorem}\label{Teig}
Fix $\alpha > 0$. Let $(W_n)_{n \geq 1} \subset \W$ be uniformly bounded
in $L^\infty$ by $\alpha$, and suppose $\boxnorm{W_n-W_0} \to 0$ as $n
\to \infty$. Denote by $\lambda_1(W_n), \lambda_2(W_n), \dots$ the
sequence of nonzero eigenvalues of $W_n$ in decreasing absolute value.
For $k \geq 1$, let $P_k(W_n): L^2([0,1]) \to L^2([0,1])$ denotes the
projection on the eigenspace of $W_n$ associated to $\{\lambda_k(W_n),
-\lambda_k(W_n)\}$. Define $\lambda_k(W_0)$ and $P_k(W_0)$ similarly.
Then for all $k \geq 1$, 
\begin{enumerate}
\item $\lambda_k(W_n) \to \lambda_k(W_0)$ as $n \to \infty$; 
\item $\|P_k(W_n) - P_k(W_0)\|_{\mathcal{B}(L^2([0,1])} \to 0$ as $n \to
\infty$, where
$\|T\|_{\mathcal{B}(L^2[0,1])} := \sup_{\|f\|_2 = 1} \|Tf\|_2$
denotes the operator norm of $T: L^2([0,1]) \to L^2([0,1])$.
\end{enumerate}

\noindent In particular, suppose for each $k \geq 1$ that
$\lambda_k(W_0)$ is a simple eigenvalue with associated eigenvector $f_0
\in L^2([0,1])$, and let $f_n \in L^2([0,1])$ be an eigenvector
associated to $\lambda_k(W_n)$. Define $p_n(x,y) := f_n(x) f_n(y)$, and
$p_0(x,y) := f_0(x) f_0(y)$. Then $\lambda_k(W_n)$ is simple for $n$
large enough and $\|p_n - p_0\|_{L^2([0,1]^2)} \to 0$ as $n \to \infty$. 
\end{theorem}

\begin{proof}
The first part of the theorem is \cite[Theorem 11.54]{Lovasz:2013}. To
prove the second part, recall that the only accumulation point of the
eigenvalues of a compact self-adjoint operator is $0$. Let $n_1 < n_2 <
\cdots$ be the set of indices such that $|\lambda_{n_k}(W_0)| \ne
|\lambda_{n_{k+1}}(W_0)|$. Define $\alpha_{k} := (|\lambda_{n_k}(W_0)| +
|\lambda_{n_{k+1}}(W_0)|)/2$ for $k \geq 1$. By part (1) and Theorem
\ref{Tszegedy1}, which we can apply by the uniform boundedness condition
on $W_n$, we have that $P_{n_1}(W_n) = [W_n]_{\alpha_1} /
\lambda_{n_1}(W_n)$ converges to $P_{n_1}(W_0) = [W_0]_{\alpha_1} /
\lambda_{n_1}(W_0)$ in operator norm on $L^2$. Similarly, for $k \geq 1$,
\begin{equation*}
P_{n_k}(W_n) = \frac{1}{\lambda_{n_k}(W_n)}\left([W_n]_{\alpha_{k+1}} -
[W_n]_{\alpha_k}\right) \to
\frac{1}{\lambda_{n_k}(W_0)}\left([W_0]_{\alpha_{k+1}} -
[W_0]_{\alpha_k}\right) = P_{n_k}(W_0)
\end{equation*}
in the operator norm on $L^2$ as $n \to \infty$.  
\end{proof}

\begin{lemma}\label{Lproj2f}
Let $\{f_i\}_{i \geq 1} \subset L^2([0,1])$ and fix $g \in
L^\infty([0,1])$. Suppose $f_i(x) f_i(y) \to f(x) f(y)$ in
$L^2([0,1]^2)$. Moreover, suppose the $f_i$ and $f$ are normalized so
that $\langle f_i, g\rangle > 0$ and $\langle f, g\rangle > 0$. Then $f_i
\to f$ in $L^2([0,1])$.
\end{lemma}

\begin{proof}
By Jensen's inequality,
\begin{align*}
\int_0^1 \left(\int_0^1 \left[f_i(x) f_i(y) - f(x)f(y)\right] g(y)\ dy
\right)^2 dx &\leq \int_0^1 \int_0^1 \left(\left[f_i(x) f_i(y) -
f(x)f(y)\right] g(y)\right)^2\ dxdy \\ 
&\leq \|g\|_\infty^2 \|f_i(x)f_i(y) - f(x)f(y)\|_2 \\
&\to 0 
\end{align*}

\noindent as $i \to \infty$. Thus, 
\begin{equation*}
\int_0^1 f_i(x) f_i(y) g(y)\ dy \to \int_0^1 f(x) f(y) g(y)\ dy
\end{equation*}
in $L^2([0,1])$. Equivalently,
\begin{equation*}
\langle f_i, g\rangle f_i(x) \to \langle f, g\rangle f(x) \quad \textrm{
in } L^2([0,1]). 
\end{equation*}

\noindent Using a similar argument, we conclude that 
\begin{equation*}
\langle f_i, g\rangle^2 = \int_0^1 \int_0^1 f_i(x) f_i(y) g(x) g(y)\ dxdy
\to \int_0^1 \int_0^1 f(x) f(y) g(x)g(y)\ dxdy = \langle f, g\rangle^2. 
\end{equation*}

\noindent Since by assumption $\langle f_i, g\rangle > 0$ and $\langle f,
g\rangle > 0$, it follows that $f_i \to f$ in $L^2([0,1])$. 
\end{proof}

\subsection{Convergence of Laplacians}

Given a graphon $W \in \W_{[0,1]}$, let $d: [0,1] \to \R$ denote its
\textit{degree function}: 
\[
d(x) := \int_0^1 W(x,y)\ dy.
\]

\noindent We identify $d$ with a multiplication operator $M_d:
L^\infty([0,1]) \to L^1([0,1])$ defined by 
\[
M_d(f)(x) = d(x) f(x). 
\]

\begin{definition}
We define the \textit{Laplacian} of $W \in \W_{[0,1]}$ to be the operator
$L_W: L^\infty([0,1]) \to L^1([0,1])$ given by 
\begin{equation}
L_W := M_d - T_W.
\end{equation}
\end{definition}

The next lemma shows that the corresponding sequence of Laplacians of a
convergent sequence of graphons is convergent in the $L^\infty \to L^1$
operator norm.

\begin{lemma}\label{Llaplace}
Let $(W_n)_{n \geq 1} \subset \W_{[0,1]}$ and $W_0 \in \W_{[0,1]}$.
Suppose $\boxnorm{W_n-W_0} \to 0$ as $n \to \infty$. Then $\|L_{W_n} -
L_{W_0}\|_{\infty \to 1} \to 0$ as $n \to \infty$.
\end{lemma}

\begin{proof}
We have 
\begin{align*}
\|L_{W_n} - L_{W_0}\|_{\infty \to 1} &=  \|M_{d_n} - T_{W_n} - M_{d_0} +
T_{W_0}\|_{\infty \to 1} \\
&\leq \|M_{d_n} - M_{d_0}\|_{\infty \to 1} + \|T_{W_0} -
T_{W_n}\|_{\infty \to 1} \\
&\leq \|M_{d_n} - M_{d_0}\|_{\infty \to 1} + 4 \boxnorm{T_{W_0} -
T_{W_n}}.
\end{align*}

\noindent It therefore suffices to show that $\|M_{d_n} -
M_{d_0}\|_{\infty \to 1} \to 0$ as $n \to \infty$. Now, 
\begin{align*}
\|M_{d_n} - M_{d_0}\|_{\infty \to 1} &= \sup_{\|f\|_\infty = 1}
\|M_{d_n}(f) - M_{d_0}(f)\|_1 \\
&=  \sup_{\|f\|_\infty = 1} \int_0^1 \left|(d_n(x)-d_0(x)) f(x) \right|\
dx \\
&= \|d_n-d_0\|_1 \\
&= \int_0^1 \left|\int_0^1 [W_n(x,y) - W_0(x,y)]\ dy\right|\ dx \\
&= \|(T_{W_n} - T_{W_0})(1)\|_1 \\
&\leq \|T_{W_n} - T_{W_0}\|_{\infty \to 1} \\
&\leq 4 \boxnorm{W_n - W_0}.
\end{align*}
It follows that $\|L_{W_n} - L_{W_0}\|_{\infty \to 1}$ as $n \to \infty$.
\end{proof}

As Theorem \ref{Teig} shows, if $\boxnorm{W_n - W_0} \to 0$, then the
eigenvalues and eigenvectors of $W_n$ converge to those of $W_0$.
However, the same result does not hold in general if $T_n: L^\infty \to
L^1$ is an arbitrary sequence of operators such that  $\|T_n -
T_0\|_{\infty \to 1} \to 0$. In particular, even though
$\|L_{W_n}-L_{W_0}\|_{\infty \to 1} \to 0$ when $\boxnorm{W_n - W_0} \to
0$, in general the eigenvalues and eigenvectors of $L_{W_n}$ may not
converge. Such a phenomenon was previously observed in
\cite{vonLuxburg2008}, where it is shown problems can occur with
unnormalized clustering in examples which are highly relevant to
practical applications (see Result 3 and the subsequent discussion, as
well as Section 8.2, in \cite{vonLuxburg2008}).

We now provide a family of examples to illustrate some of the problems
that can occur in our framework, when working with the unnormalized
Laplacian. We first recall some preliminaries on the essential spectrum.
Let $X$ be a Hilbert space, and
denote the spectrum of an operator $T: X \to X$ by 
\[
\sigma(T) := \{\lambda \in \mathbb{C} : T - \lambda I \textrm{ is not
invertible}\}.
\]

\noindent Recall that the \textit{discrete spectrum} of $T$, denoted
$\sigma_{\textrm{discr}}(T)$, is the set of isolated eigenvalues of $T$
of finite multiplicity. Denote by $\sigma_{\textrm{ess}}(T) := \sigma(T)
\setminus \sigma_{\textrm{discr}}(T)$ the \textit{essential spectrum} of
$T$; this is a closed subset of $\mathbb{C}$. Moreover,
$\sigma_{\textrm{ess}}(T+K) = \sigma_{\textrm{ess}}(T)$ for any compact
operator $K$, i.e., the essential spectrum is closed under compact
perturbation. Recall that $\lambda \in \sigma(T)$ if and only if there
exists a sequence $(\psi_k)_{k \geq 1} \subset X$ such that $\|\psi_k\| =
1$ for all $k$, and
\[
\lim_{k \to \infty} \|T \psi_k - \lambda \psi_k\| = 0.
\]

\noindent Moreover, $\lambda \in \sigma_{\textrm{ess}}(T)$ if such a
sequence $(\psi_k)_{k \geq 1}$ with no convergent subsequence exists. See
\cite{spectral-book} for more details on the essential spectrum.

\begin{prop}\label{Phyper}
Given a continuous function $g: [0,1] \to [0,1]$, define the graphon
$W_g$ by:
\[
W_g(x,y) := g(x) g(y), \qquad x,y \in [0,1].
\]

\noindent Let $E_0 := g^{-1}(0)$. Then the only eigenvalue for $L_{W_g}$
is $0$, with eigenspace given by
\[
\ker L_{W_g} = \{ f \in L^2([0,1]) : f \text{ is constant on } [0,1]
\setminus E_0 \}.
\]

\noindent However, $\sigma_{\textrm{ess}}(L_{W_g}) = d([0,1])$ where
$d(x) := \langle g, {\bf 1}\rangle g(x)$ denotes the degree function of
$W_g$.
\end{prop}

\begin{proof}
In this proof, let ${\bf 1}$ denote the constant function $1$ on $[0,1]$,
and denote by $\tangle{g,f}$ the inner product of $g$ with $f \in
L^2([0,1])$.

Notice that $W_g$ has degree function $d(y) = \tangle{g,{\bf 1}} g(y)$.
Now if $(L_{W_g} f)(y) \equiv \lambda f(y)$ on $[0,1]$ for some
eigenfunction $f \in L^2([0,1])$, then
\[
f(y) ( \tangle{g,{\bf 1}} g(y) - \lambda ) = (T_{W_g} f)(y) = g(y)
\tangle{g,f}.
\]

\noindent First notice, if $g \equiv 0$ then $\lambda = 0$ (since $f
\not\equiv 0$) and the result is easily shown. Thus, we assume henceforth
that $g \not\equiv 0$, whence $\tangle{g, {\bf 1}} > 0$. Now solve for
$f$ to obtain:
\begin{equation}\label{Ehyper1}
f(y) = \frac{\tangle{g,f} g(y)}{\tangle{g,{\bf 1}} g(y) - \lambda}.
\end{equation}

\noindent Notice that $\tangle{g,f} \neq 0$ since $f \not\equiv 0$.

There are now several cases. If $\lambda = 0$ then $f$ is constant on
$[0,1] \setminus E_0$ (and \textit{a priori} arbitrary on $E_0$). In this
case the result is not hard to show.

We now show that the remaining values of $\lambda$ cannot be eigenvalues
for $L_{W_g}$. Indeed, first suppose $\lambda / \tangle{g,{\bf 1}} \in
g([0,1])$; then the preceding equation shows that $f \not\in L^2([0,1])$,
since $\lambda \neq 0$.
For all other values of $\lambda$, i.e.~$\lambda \not\in \tangle{g,{\bf
1}} g([0,1]) \cup \{ 0 \}$, evaluate both sides of Equation
\eqref{Ehyper1} against $g(y) \tangle{g,{\bf 1}}^2 / \tangle{g,f}$ and
compute:
\[
\tangle{g,{\bf 1}}^2 = \int_0^1 \frac{\tangle{g,{\bf 1}}^2
g(y)^2}{\tangle{g,{\bf 1}} g(y) - \lambda}\ dy = \int_0^1 (\tangle{g,{\bf
1}} g(y)+ \lambda)\ dy + \int_0^1 \frac{\lambda^2\ dy}{\tangle{g,{\bf 1}}
g(y) - \lambda}.
\]

\noindent Cancel $\tangle{g,{\bf 1}}^2$, and simplify using that $\lambda
\neq 0$, to obtain:
\begin{equation}\label{Ehyper2}
1 = \int_0^1 \frac{\lambda\ dy}{\lambda - \tangle{g,{\bf 1}} g(y)}.
\end{equation}

There are now three cases. First if $\lambda < 0$, then since $\sup_y
g(y) > 0$ and $g$ is continuous, hence the integrand always lies in
$[0,1]$, and is bounded above by
\[
\frac{|\lambda|}{|\lambda| + \tangle{g,{\bf 1}} (\sup_y g(y)/2)} \in
(0,1)
\]

\noindent on a set of positive measure. Therefore it cannot integrate to
$1$ on $[0,1]$.

The remaining cases are similar. First suppose $\lambda / \tangle{g,{\bf
1}} \in (0,\inf_y g(y))$ (assuming $g(y)$ is always positive on $[0,1]$).
In this case, the integrand in Equation \eqref{Ehyper2} is always
negative, which is impossible. The only other possibility for $\lambda$
is $\lambda / \tangle{g,{\bf 1}} > \sup_y g(y)$, since $g([0,1])$ is an
interval by the continuity of $g$. In this case, the integrand always
lies in $[1,\infty)$, and is bounded below by
\[
\frac{\lambda}{\lambda - \tangle{g,{\bf 1}} (\sup_y g(y)/2)} \in
(1,\infty)
\]

\noindent on a set of positive measure. Therefore it cannot integrate to
$1$ on $[0,1]$. It follows that $\lambda = 0$ is the only eigenvalue for
the Laplacian of $W_g$.

Finally, since $T_{W_g}$ is compact, $\sigma_{\textrm{ess}}(L_{W_g}) =
\sigma_{\textrm{ess}}(M_d)$. Clearly, $\sigma(M_d) \subset d([0,1])$.
Now, let $\lambda \in d([0,1])$, say $\lambda = d(x_0)$. Since $g$ is
continuous, there exists a sequence $\epsilon_k \to 0$ such that $|d(x)
-d(x_0)| < \epsilon_k$ if $|x-x_0| < 1/k$. Now define
\[
I_k := B(x_0, 1/k) \cap [0,1], \qquad \psi_k(x) :=
\frac{1}{\sqrt{\mu_L(I_k)}} {\bf 1}_{x \in I_k}.
\]

\noindent Note that $\|\psi_k\|_2 = 1$. Now,
\[
\|M_d \psi_k - \lambda \psi_k\|_2^2 =
\int_0^1 [(d(x) - d(x_0))\psi_k(x)]^2\ dx
\leq \epsilon_k^2 \cdot \|\psi_k\|_2^2 = \epsilon_k^2 \to 0
\]
as $k \to \infty$. Clearly $(\psi_k)_{k \geq 1}$ has no convergent
subsequence in $L^2([0,1])$. Therefore, $\lambda \in
\sigma_{\textrm{ess}}(M_d)$.
\end{proof}

Proposition \ref{Phyper} above as well as the work in
\cite{vonLuxburg2008} demonstrate that significant problems can occur
when working with the unnormalized Laplacian in clustering applications.
We now study in detail the properties of the normalized Laplacian, and
prove the structural consistency of the resulting clustering algorithm
under broad assumptions.

\subsection{The normalized Laplacian}\label{SnormLap}

As in the unnormalized Laplacian case, we now extend the normalized
Laplacian of graphs to graphons. Recall that the normalized Laplacian
$L'_G$ of a graph $G$ with $n$ vertices is given by 
\[
L'_G = D^{-1/2} L_G D^{-1/2} = I - D^{-1/2} A D^{-1/2}, 
\]

\noindent where $D = \textrm{diag}(d_1, \dots, d_n)$ is a diagonal matrix
with the degrees of the vertices on the diagonal, and $A$ is the
adjacency matrix of $G$. The normalized Laplacian naturally arises in
spectral clustering when relaxing the Normalized Cut problem instead of
the Ratio Cut problem (see \cite[Section 5]{vonLuxburg-tutorial} for more
details). Akin to the unnormalized Laplacian, we extend the definition of
the normalized Laplacian by viewing it as an operator from $L^\infty$ to
$L^1$.

\begin{definition}\label{DnLap}
Let $W \in \W_{[0,1]}$ be a graphon with degree function $d$.  We
define the \textit{normalized} kernel $W'$ by 
\[
W'(x,y) := \begin{cases}
\frac{W(x,y)}{\sqrt{d(x) d(y)}}, & \textrm{ if } d(x), d(y) \ne 0, \\
0, & \textrm{ otherwise}.
\end{cases}
\]

\noindent We define the \textit{normalized Laplacian} of $W$ to be the
operator $L_W' : L^\infty([0,1]) \to L^1([0,1])$ given by
\begin{equation}
L_W' := M_{{\bf 1}_{d(x) \neq 0}} - T_{W'}.
\end{equation}
\end{definition}

Note that $W'$ is not necessarily bounded. However, as we now show, if $W
\in \W_{[0,1]}$, then $\boxnorm{W'} = \|W'\|_1$ is uniformly bounded.
We will deduce this from the following technical lemma, which will also
be useful later for analyzing the convergence of normalized Laplacians.

\begin{lemma}\label{LPeterTrick}
Let $W \in \W_{[0,1]}$ with degree function $d$, and let $W'$ denote the
associated normalized kernel as in Definition \ref{DnLap}. Then for
every measurable $A,B \subset [0,1]$, 
\begin{equation}\label{EPeterTrickB}
\int_{A \times B} W'(x,y)\ dxdy \leq \mu_L(A)^{1/2} \cdot \mu_L(B)^{1/2},
\end{equation}

\noindent and moreover,
\begin{equation}\label{EPeterTrickA}
\int_{A \times B} W'(x,y)\ dxdy \leq 2 \min(\mu_L(A), \mu_L(B)).
\end{equation}
\end{lemma}

\begin{proof}
To prove the inequality \eqref{EPeterTrickB}, let $P := \{x \in [0,1] :
d(x) > 0\}$. Then 
\[
\int_{A \times B} W'(x,y)\ dxdy = \int_{(A \cap P) \times (B \cap P)}
\frac{W(x,y)}{\sqrt{d(x) d(y)}}\ dxdy.
\]

\noindent By the Cauchy--Schwarz inequality, 
\begin{align*}
&\int_{(A \cap P) \times (B \cap P)} \frac{W(x,y)}{\sqrt{d(x) d(y)}}\
dxdy = \int_0^1 \int_0^1 \sqrt{\frac{W(x,y)}{d(x)}} {\bf 1}_{A \cap P}(x)
\sqrt{\frac{W(x,y)}{d(y)}}{\bf 1}_{B \cap P}(y)\ dxdy \\
&\leq \left(\int_0^1 \int_0^1 \frac{W(x,y)}{d(x)} {\bf 1}_{A \cap P}(x)\
dxdy\right)^{1/2} \left(\int_0^1 \int_0^1 \frac{W(x,y)}{d(y)} {\bf 1}_{B
\cap P}(y)\ dxdy\right)^{1/2} \\
&= \left(\int_{A \cap P} \frac{d(x)}{d(x)}\ dx\right)^{1/2} \left(\int_{B
\cap P} \frac{d(y)}{d(y)}\ dy\right)^{1/2} \\
&\leq \mu_L(A)^{1/2} \cdot \mu_L(B)^{1/2}.
\end{align*}

To prove the inequality \eqref{EPeterTrickA}, we may assume without loss
of generality that the degree function $d$ is non-decreasing (otherwise,
replace $W$ by $W^\sigma$ for an appropriate $\sigma \in S_{[0,1]}$).
Note that if $d(x) = 0$, then $W'(x,y) = 0$ for almost every $y \in
[0,1]$. Thus,
\[
\int_{A \times B} W'(x,y)\ dxdy = \int_{(A \cap P) \times B} W'(x,y)\
dxdy = \int_{(A \cap P) \times B} \frac{W(x,y)}{\sqrt{d(x)d(y)}}\ dxdy.
\]

\noindent where $P$ was defined above. Now, 
\begin{align*}
\int_{(A \cap P) \times B} \frac{W(x,y)}{\sqrt{d(x)d(y)}}\ dxdy &\leq
\int_{A \cap P} \int_0^1 \frac{W(x,y)}{\sqrt{d(x)d(y)}}\ dxdy \\
&= 2 \int_{y \in A \cap P} \int_y^1  \frac{W(x,y)}{\sqrt{d(x)d(y)}}\ dxdy
\\
&\leq 2 \int_{y \in A \cap P} \frac{1}{d(y)} \int_y^1 W(x,y)\ dxdy \\
&\leq 2 \int_{y \in A \cap P} \frac{1}{d(y)} \int_0^1 W(x,y)\ dxdy \\
&= 2 \cdot \mu_L (A \cap P) \\
&\leq 2 \cdot \mu_L(A).
\end{align*}
The result follows by carrying out a similar computation using $B \cap P$
instead of $A \cap P$.
\end{proof}

\begin{cor}\label{PboundWp}
Let $W \in \W_{[0,1]}$ with degree function $d$, and let $W'$ be the
normalized kernel associated to $W$ as in Definition \ref{DnLap}. Then 
\[
\boxnorm{W'} \leq 1, 
\]
Moreover, the bound is sharp.
\end{cor}

\begin{proof}
The bound $\boxnorm{W'} \leq 1$ follows immediately from Equation
\eqref{EPeterTrickB}. Using $W \equiv 1$, it follows that the bound is
sharp.
\end{proof}

\begin{cor}
The operator $L_W': L^\infty([0,1]) \to L^1([0,1])$ is bounded. 
\end{cor}

\begin{proof}
Recall that $L_W' =  M_{{\bf 1}_{d(x) \neq 0}} - T_{W'}$ and that by
Corollary \ref{PboundWp}, $W' \in L^1([0,1]^2)$. Thus, by Fubini's
theorem, $L_W'$ is a well-defined operator from $L^\infty$ to $L^1$.
The operator being bounded follows from from the fact that the operator
norm of $T_{W'}$ is equivalent to the cut-norm of $W'$ (see Equation
\eqref{Ecutinfty1}).
\end{proof}

In the remainder of this subsection, our goal is to understand when the
normalized Laplacians converge. Doing so will allow us to derive the
convergence of the eigenvectors and eigenvalues, and conclude convergence
of the spectral clustering derived from those Laplacians.
The only assumption that we will need is $d_0 > 0$; in other words, the
limit graphon does not have any isolated sparse regions. In the case of
finite graphs this is just asking for no isolated nodes.
The following example explains why the assumption that $d_0 > 0$ is
indeed necessary going forward in the paper. However, once we make this
assumption, we are able to show convergence of the normalized Laplacian
in full generality. See Theorem \ref{Tconv_norm_lap}.

\begin{ex}\label{EnormLap}
Consider first the case $W_0 = 0$.  Then any sparse sequence of graphons
$W_n$ converges to $W_0$ in cut-norm.  For example, let $\{P_i\}_{i
=1}^m$ be a partition of $[0,1]$.  Let $U = \sum_{i =1}^m {\bf 1}_{P_i
\times P_i}$.  Let $U^\delta$ be a small perturbation of $U$ so that it
maintains the same block diagonal structure as $U$ but has $m$ simple
eigenvalues in $(1 - \delta, 1)$ with eigenvectors that are small
perturbations of ${\bf 1}_{P_i}$.   Then $W_n = \frac{1}{n}U^\delta \to
W_0$ in cut-norm.  Now, $L'_{W_n} = L'_{W_1} = I - T_{W'_1}$ for all $n$.
Therefore, the limit coloring will be defined by the partitions
$\{P_i\}_{i = 1}^m$.  As this procedure applies for arbitrary partitions
of the unit interval $[0,1]$, any partition can be derived from $W_n \to
0$ and spectral clustering fails to be well-defined in the limit.  In
addition, the same argument shows that the normalized Laplacians do not
converge when the underlying graphons converge.

More generally, let $W_0$ be a general graphon whose degree function
$d_0$ takes the value $0$ on a set of positive measure. Given $W_0$, we
may permute it by $\sigma \in S_{[0,1]}$ and assume without loss of
generality that $d_0(x) = 0$ if and only if $x \in E_0 := [0,a]$ for some
$a \in [0,1]$. Let $W^{+}_0$ denote the graphon restricted to $E_0^c
\times E_0^c$; thus the graphon is of the form $\begin{pmatrix} 0 & 0\\0
& W^{+}_0\end{pmatrix}$

Let the normalized kernel be $W'_0 = \begin{pmatrix} 0 & 0\\ 0 &
W^{+'}_0\end{pmatrix}$. Under the assumption $W^{+'}_0 \in L^2$, we can
say that $W'_0$ is a symmetric, Hilbert--Schmidt operator, with a
discrete spectrum of real eigenvalues and only possible accumulation
point at $0$.    

Now let $\{P_i\}_{i =1}^m$ be a partition of $[0,a]$, and define $U :=
\sum_{i =1}^m {\bf 1}_{P_i \times P_i}$.  Let $U^\delta$ be a small
perturbation of $U$ so that it maintains the same block diagonal
structure as $U$, but has $m$ simple eigenvalues in $(1 - \delta, 1)$
with eigenvectors that are small perturbations of ${\bf 1}_{P_i}$.  This
time we pick $\delta > 0$ to be less than the gap between $1$ and the
next largest eigenvalue of $W'_0$.  Let $W_n = \begin{pmatrix}
\frac{1}{n}U^\delta & 0\\0 & W^{+}_0\end{pmatrix}$, then $W_n \to W_0$.
Once again, because of the normalization occurring in $L'_{W_n}$, we find
that the limit clusters will be defined by the $\{P_i\}_{i =1}^m$
irrespective of the structure of the dense part $W^{+}_0$ of $W_0$.  
\end{ex}

\begin{utheorem}\label{Tconv_norm_lap}
Let $(W_n)_{n \geq 1} \subset \W_{[0,1]}$ such that $\boxnorm{W_n-W_0}
\to 0$.  Let $d_n$ and $d_0$ denote the degree functions of $W_n$ and
$W_0$ respectively.  Assume that $d_0(x) > 0$ for almost every $x$.
Define $W_n'$ and $W_0'$ as in Definition \ref{DnLap}. Then 
\[
\boxnorm{W_n' - W_0'} \to 0 \qquad \textrm{ as } n \to \infty.
\] 
Moreover, we have $\|L'_{W_n}-L'_{W_0}\|_{\infty \to 1} \to 0$ as $n
\to \infty$.
\end{utheorem}

\begin{proof}
Let $\epsilon > 0$.  Since $d_0(x) > 0$ for almost every $x$, we can pick
$\lambda \in (0,1]$ be such that $\mu_L(d_0^{-1}([0,2\lambda))) <
\epsilon$. Suppose $d_n(x) < \lambda$. Then either $x \in
d_0^{-1}([0,2\lambda))$, or $x \in \{w \in [0,1] : |d_n(w) - d_0(w)| >
\lambda\}$. Since $d_n  \to d_0$ in $L^1$ (see the proof of Lemma
\ref{Llaplace}), there exists $N_\epsilon$ such that for $n \geq
N_\epsilon$,
\begin{equation}\label{Emeasure}
\mu_L \{x \in [0,1] : |d_n(x) - d_0(x)| \geq \lambda\} \leq \epsilon. 
\end{equation}
It follows that for $n \geq N_\epsilon$, 
\begin{equation}\label{Elambda}
\mu_L (d_n^{-1}([0,\lambda))) \leq 2\epsilon.
\end{equation}
Now, for $n \geq 0$, define 
\[
P_n := \{x \in [0,1] : d_n(x) > 0\}, 
\]

\noindent and let $Z := [0,1]^2 \setminus (P_n \times P_n)$. We claim
that for $n \geq N_\epsilon$,  
\begin{equation}\label{Estep1}
\boxnorm{(W_n'-W_0'){\bf 1}_{Z}} \leq 8\epsilon.
\end{equation}

\noindent Indeed, if $d_n(x) = 0$, then $W_n'(x,y) = 0$ for almost all $y
\in [0,1]$ and so, for $A \subset P_n^c$ and $B \subset [0,1]$, 
\begin{align*}
\left|\int_{A \times B} [W_n'(x,y) - W_0'(x,y)]\ dxdy\right| &\leq
\int_{P_n^c \times [0,1]} W_n'(x,y)\ dxdy \\
&= \int_{P_n^c \times [0,1]}
\frac{W_n(x,y)}{\sqrt{d_n(x)d_n(y)}}\ dxdy.
\end{align*}

\noindent For $n \geq N_\epsilon$, we have by \eqref{Emeasure} that
$\mu_L (P_n^c) \le 2\epsilon.$ It follows by Equation
\eqref{EPeterTrickB} that 
\[
\left|\int_{A \times B} [W_n'(x,y) - W_0'(x,y)]\ dxdy\right| \leq
4\epsilon.
\]
This proves \eqref{Estep1}.

We will now prove that $\boxnorm{(W_n'-W_0'){\bf 1}_{P_n \times P_n}}
\leq 8\epsilon$ for $n$ is large enough. To do so, define: 
\begin{center}
\begin{tabular}{lll}
$Q_n :=  d_0^{-1}((0,\lambda)) \cup d_n^{-1}((0,\lambda))$, & &  $R_n :=
P_n \setminus Q_n$, \\  
$S_n := R_n \times R_n$, & & $T_n := (P_n \times P_n) \setminus S_n$.
\end{tabular}
\end{center}

\noindent We will first prove that $\boxnorm{(W_n'-W_0'){\bf 1}_{T_n}}
\leq 16 \epsilon$ if $n \geq N_\epsilon$. Indeed, If $A,B \subset
T_n$, then 
\begin{align*}
\left|\int_{A \times B} [W_n'(x,y) - W_0'(x,y)]\ dxdy\right| &=
\left|\int_{A \times B} \left(\frac{W_n(x,y)}{\sqrt{d_n(x)d_n(y)}} -
\frac{W_0(x,y)}{\sqrt{d_0(x)d_0(y)}}\right)\ dx dy \right| \\
&\leq \int_{T_n} \frac{W_n(x,y)}{\sqrt{d_n(x)d_n(y)}}\ dxdy + \int_{T_n}
\frac{W_0(x,y)}{\sqrt{d_0(x)d_0(y)}}\ dxdy\\
&\leq 8 \epsilon + 8 \epsilon = 16 \epsilon.
\end{align*}
where the last inequality was obtained by Equations \eqref{EPeterTrickA}
and \eqref{Elambda}.

Finally, we show that $\boxnorm{(W_n'-W_0'){\bf 1}_{S_n}} \leq \epsilon
(2/\sqrt{\lambda} + 1)$ for $n$ large enough. Note that for $x \in R_n$,
we have $|d_n(x)^{-1/2} - d_0(x)^{-1/2}| \le 1/\lambda^{1/2}$. Since
$x^{-1/2}$ is Lipschitz in $[\lambda, 1]$ with Lipschitz constant $C$ for
some $C > 0$, we have
\[
\int_{R_n}|d_n^{-1/2}(x) - d_0^{-1/2}(x)| dx \leq C \cdot \int_{R_n}
|d_n(x) - d_0(x)| < \epsilon
\] 

\noindent for $n \ge M_{\epsilon}$ since $d_n \to d_0$ in $L^1$. Now, for
$x,y \in P_0 \cap P_n$, we have
\begin{align*}
& W_n'(x,y)-W_0'(x,y)\\
& = \frac{W_n(x,y)}{\sqrt{d_n(x)d_n(y)}} -
\frac{W_0(x,y)}{\sqrt{d_0(x)d_0(y)}} \\
& = \left(\frac{W_n(x,y)}{\sqrt{d_n(x)d_n(y)}} -
\frac{W_n(x,y)}{\sqrt{d_n(x)d_0(y)}}\right) +
\left(\frac{W_n(x,y)}{\sqrt{d_n(x)d_0(y)}} -
\frac{W_0(x,y)}{\sqrt{d_n(x)d_0(y)}}\right) \\
& \quad + \left(\frac{W_0(x,y)}{\sqrt{d_n(x)d_0(y)}} -
\frac{W_0(x,y)}{\sqrt{d_0(x)d_0(y)}}\right).
\end{align*}

\noindent We will bound the integral of each term separately. 

First, for $A,B \subset R_n$, 
\begin{align*}
\left|\int_{A \times B} \left(\frac{W_n(x,y)}{\sqrt{d_n(x)d_n(y)}} -
\frac{W_n(x,y)}{\sqrt{d_n(x)d_0(y)}}\right) dxdy\right| &\leq \int_{A
\times B} \left|\frac{1}{\sqrt{d_n(y)}} - \frac{1}{\sqrt{d_0(x)}}\right|
\frac{W_n(x,y)}{\sqrt{d_n(x)}} dxdy \\
&\leq \epsilon \cdot \frac{1}{\sqrt{\lambda}}.
\end{align*}

\noindent For the second term, we have 
\begin{align*}
& \left|\int_{A \times B} \left(\frac{W_n(x,y)}{\sqrt{d_n(x)d_0(y)}} -
\frac{W_0(x,y)}{\sqrt{d_n(x)d_0(y)}}\right)\ dxdy\right|\\
& = \left|\int_{A \times B} \left(W_n(x,y) - W_0(x,y) \right)
\frac{1}{\sqrt{d_n(x)d_0(y)}}\ dxdy\right| \\
& \leq \frac{1}{\lambda} \sup_{\substack{f, g: [0,1] \to [0,1] \\
{\rm supp}\ f,g \subset R_n}} \int_{[0,1]^2} \left(W_n(x,y) - W_0(x,y)
\right) f(x) g(y)\ dxdy \\
& \leq \frac{1}{\lambda} \boxnorm{(W_n-W_0){\bf 1}_{S_n}} \leq
\frac{1}{\lambda} \boxnorm{W_n-W_0} \leq \epsilon
\end{align*}
for $n \geq N'_\epsilon$.

For the third term, we have
\begin{align*}
& \left|\int_{A \times B}\left(\frac{W_0(x,y)}{\sqrt{d_n(x)d_0(y)}} -
\frac{W_0(x,y)}{\sqrt{d_0(x)d_0(y)}}\right)\ dxdy \right|\\
& \leq \frac{1}{\sqrt{\lambda}} \int_{A \times B}
\left|\frac{1}{\sqrt{d_n(x)}} - \frac{1}{\sqrt{d_0(x)}} \right| W_0(x,y)\
dxdy \\
& \leq \epsilon \cdot \frac{1}{\sqrt{\lambda}}.
\end{align*}

\noindent Putting everything together, we conclude that 
\[
\boxnorm{(W_n'(x,y)-W_0'(x,y)){\bf 1}_{S_n}} \leq 2 \epsilon \cdot
\frac{1}{\sqrt{\lambda}} + \epsilon
\]

\noindent for $n \geq \max(N_\epsilon, M_\epsilon, N'_\epsilon)$.
Finally, combining all the inequalities, we obtain that 
\[
\boxnorm{W_n'-W_0'} \leq (25 + 2/\sqrt{\lambda})\cdot \epsilon. 
\]

\noindent We conclude that $W_n' \to W_0'$ in cut-norm. To complete the
proof of the theorem, note that for $n \geq \max(N_\epsilon, M_\epsilon,
N'_\epsilon)$,
\begin{align*}
||L_{W_n}' - L_{W_0}'||_{\infty \to 1} &\le ||M_{{\bf 1}_{d_n(x) \neq 0}}
- M_{{\bf 1}_{d_0(x) \neq 0}}||_{\infty \to 1} + ||T_{W'_n} -
T_{W'_0}||_{\infty \to 1}.\\
&\le \mu_L(\{x : d_n(x) = 0\})  + 4 \boxnorm{W_n'-W_0'} \\
&\le 2 \epsilon  + 4 \boxnorm{W_n'-W_0'}
\end{align*}

\noindent by \eqref{Emeasure}, and this approaches $0$.
\end{proof}

We conclude this subsection by observing that the usual eigenvalue bounds
hold for our generalized version of the normalized Laplacian:

\begin{prop}\label{Pnorm-lap-evalue}
Given $W \in W_{[0,1]}$ with degree function $d$, all eigenvalues of the
normalized Laplacian $L'_W$ lie in $[0,2]$.
\end{prop}

\begin{proof}
We provide a proof-sketch for completeness, as our setting is slightly
more general than is usually found in the literature (although the
argument is more or less standard). Define $D_+ := \{ x \in [0,1] : d(x)
> 0 \}$; now if $L'_W g = \lambda g$ for some (nonzero) eigenfunction
$g$, then evaluating against $g$ yields:
\[
\lambda \int_0^1 g(x)^2\ dx = \int_0^1 g(x)^2 d(x) {\bf 1}_{d(x) > 0}\ dx
- \int_{D_+} \int_{D_+} g(x) \frac{W(x,y)}{\sqrt{d(x) d(y)}} g(y)\ dy\
dx.
\]

\noindent Define $h(x) := g(x)/\sqrt{d(x)}$ on $D_+$, and $0$ otherwise.
Then one verifies that the above equation translates to:
\[
\lambda \int_0^1 g(x)^2\ dx = \frac{1}{2} \int_{D_+} \int_{D_+} W(x,y)
(h(x) - h(y))^2\ dy dx \geq 0,
\]

\noindent whence $\lambda \geq 0$. On the other hand,
\begin{align*}
\frac{1}{2} \int_{D_+} \int_{D_+} W(x,y) (h(x) - h(y))^2\ dy dx \leq &\
\frac{1}{2} \int_{D_+} \int_{D_+} W(x,y) (2 h(x)^2 + 2 h(y)^2)\ dy dx\\
= &\ 2 \int_{D_+} h(x)^2 d(x)\ dx \leq 2 
\int_0^1 g(x)^2\ dx,
\end{align*}

\noindent from which it follows that $\lambda \leq 2$.
\end{proof}

\subsection{Structural consistency of spectral clustering (normalized
Laplacian)}\label{Sconsistency}

We now examine the convergence of the eigenvalues and eigenvectors of the
normalized Laplacian of a convergent sequence of graphons. 

Given a graphon $W \in \W_{[0,1]}$, we will denote by $\mu_1(W), \dots,
\mu_k(W)$ the $k$ smallest nonzero eigenvalues of $L'_{W}$ (see
Proposition \ref{Pnorm-lap-evalue}), and by $f^1_W, \dots, f^k_W \in
L^2([0,1])$ associated eigenvectors/eigenfunctions.

\begin{utheorem}\label{TSpClusteringNorm}
Fix $m \geq 1$, let $\mathcal{D}_{m,\alpha}$ be the set of graphons $W$
such that, $W' \in L^\infty([0,1]^2)$ is uniformly bounded above by
$\alpha > 0$, and $\mu_1(W), \dots, \mu_m(W)$ are all simple. Let
$(W_n)_{n \geq 1} \subset \mathcal{D}_{m,\alpha}$ and $W_0 \in
\mathcal{D}_{m,\alpha}$ such that $\boxnorm{W_n- W_0} \to 0$ as $n \to
\infty$. Assume that 
\[
\int_0^1 W_0(x,y)\ dy > 0 \qquad \textrm{for a.e. } y \in [0,1].
\] 

\noindent Normalize the associated eigenvectors of the smallest $m$
eigenvalues $\mu_1(W_n), \dots, \mu_m(W_n)$ so that 
\begin{equation*}
\langle f^1_{W_n}, h\rangle > 0, \dots, \langle f^m_{W_n}, h\rangle > 0
\end{equation*}

\noindent for some $h \in L^\infty([0,1])$ and all $n \geq 0$. Define $f
: \{ W_n : n \geq 0 \} \times [0,1] \to \R^m$ by 
\[
f(W_n,y) := (f^1_{W_n}(y), \dots, f^m_{W_n}(y))^T. 
\]

\noindent Moreover, let $N \geq 1$ and let $(A_j)_{j=1}^N \subset \R^m$
be a collection of disjoint open sets such that for all $n \geq 0$,
\[
f(W_n,y) \in \bigcup_{j=1}^N A_j \qquad \textrm{for a.e. } y \in [0,1].
\]

\noindent Then we have:
\begin{enumerate}
\item $f(W^\sigma,y) = f(W,\sigma(y))$ for all $\sigma \in S_{[0,1]}$, $W
\in \mathcal{D}_{m,\alpha}$, and almost every $y \in [0,1]$.

\item Set $S := \{1,\dots,N\}$, and for $n \geq 0$, define $F(W_n) :=
(W_n, c_{W_n})$, where $c_{W_n}(y)$ is the unique $j \in \{1,\dots,N\}$
such that $f(W_n,y) \in A_j$. Then $\boxnorm{F(W_n)- F(W_0)}^S \to 0$ as
$n \to \infty$.
\end{enumerate}
\end{utheorem}

\begin{remark}
It is useful to examine how Theorem \ref{TSpClusteringNorm} above
compares to the theorems in \cite{vonLuxburg2008} on convergence of
spectral clustering using normalized Laplacians.  Firstly, in that work,
the degree functions $d_n$ satisfy $d_n \geq \lambda > 0$. We make a more
general assumption about the normalized Laplacians (see Example
\ref{Elaplacian}).
Secondly and very importantly, we do not make a specific modeling
assumption about the data generating mechanism for the graphs. We merely
assume that the graphs come to us in a way that is convergent in the
graph topology.  This subsumes the mechanism assumed in
\cite{vonLuxburg2008} as a special case and includes many others.  As a
consequence, our arguments require different techniques which are
suitable to the graphon topology.
\end{remark}

\begin{ex}\label{Elaplacian}
We illustrate with an example how Theorem \ref{TSpClusteringNorm} allows
working outside the setting in \cite{vonLuxburg2008}, in which the degree
was assumed to be bounded away from zero. Fix $\alpha > 1$, and consider
distinct measurable functions $g_1, \dots, g_k : [0,1] \to [0,1]$ such
that
\[
\| (g_1, \dots, g_k) \|_\infty \leq 1, \qquad
\int_0^1 g_i(x)\ dx \geq \alpha^{-1}\ \forall i.
\]
Define
\[
W(x,y) := \sum_{i=1}^k g_i(x) g_i(y) \in \W_{[0,1]}.
\]

\noindent Using the notation $\tangle{\cdot,\cdot}$ for the inner product
in $L^2$, the degree function is $d_W(y) = \sum_i \tangle{g_i, {\bf 1}}
g_i(y)$, and this is not necessarily bounded away from $0$. Now the
normalized kernel is
\[
W'(x,y) = \frac{\sum_i g_i(x) g_i(y)}{\sqrt{\sum_i \tangle{g_i, {\bf 1}}
g_i(x) \cdot \sum_i \tangle{g_i, {\bf 1}} g_i(y)}}.
\]

\noindent Now note that since $g_i(y) \in [0,1]$, by choice of $\alpha$
we have
\[
\sum_{i=1}^k g_i(y)^2 \leq \sum_{i=1}^k g_i(y) \cdot \tangle{g_i, {\bf
1}} \alpha, \qquad \forall y \in [0,1].
\]

\noindent Hence by the Cauchy--Schwarz inequality,
\[
W'(x,y)^2 \leq
\frac{\sum_i g_i(x)^2}{\sum_i \tangle{g_i, {\bf 1}} g_i(x)} \cdot
\frac{\sum_i g_i(y)^2}{\sum_i \tangle{g_i, {\bf 1}} g_i(y)} \leq
\alpha^2,
\]

\noindent whence $W'(x,y) \in [0,\alpha]$, satisfying the corresponding
hypothesis in Theorem \ref{TSpClusteringNorm}.
\end{ex}

We now prove Theorem \ref{TSpClusteringNorm}.

\begin{proof}[Proof of Theorem \ref{TSpClusteringNorm}]
To verify (1), suppose $\lambda$ is an eigenvalue of $T_W$ for some
symmetric $W \in L^2([0,1]^2)$. Let $f_W \in L^2([0,1])$ be an associated
eigenfunction. Then for $\sigma \in S_{[0,1]}$, 
\[
\lambda f_W(\sigma(x)) = \int_0^1 W(\sigma(x),y) f_W(y)\ dy = \int_0^1
W^\sigma(x,y) f_W(\sigma(y))\ dy. 
\]

\noindent Therefore $\lambda$ is also an eigenvalue of $T_{W^\sigma}$
with associated eigenfunction $f_{W^\sigma}(x) = f_W(\sigma(x))$. This
proves (1). 

Now suppose $\boxnorm{W_n-W_0} \to 0$. By Theorem \ref{Tconv_norm_lap},
$\boxnorm{W_n' - W_0'} \to 0$, where $W_n'$ and $W_0'$ denote the
normalized kernels as in Equation \eqref{DnLap}.
By Proposition \ref{Pnorm-lap-evalue}, the smallest $m$ nonzero
eigenvalues $\mu_1(W_n), \dots, \mu_m(W_n)$ of $L'_{W_n}$ are in
bijection with the largest eigenvalues $\lambda_1(W'_n), \dots,
\lambda_m(W'_n)$ of $W_n'$ that are not equal to $1$.
By Theorem \ref{Teig}, these eigenvalues converge to $\lambda_1(W'_0),
\dots, \lambda_m(W'_0)$. Moreover, since $W_n \in \mathcal{D}_{m,\alpha}\
\forall n$, the eigenvectors associated to $\lambda_1(W'_n), \dots,
\lambda_m(W'_n)$ are the same as the eigenvectors of $L'_{W_n}$
associated to $1-\lambda_1(W'_n), \dots, 1-\lambda_m(W'_n)$. Now since
the $W'_n$ are uniformly bounded by $\alpha$, apply Theorem \ref{Teig}(2)
and Lemma \ref{Lproj2f} to obtain
\[
\|f^i_{W_n} - f^i_{W_0}\|_2 \to 0 \qquad \textrm{as } n \to \infty \qquad
(i=1,\dots,m).
\]

\noindent Since $[0,1]$ has finite measure, by Cauchy--Schwarz it follows
that
\[
\|f^i_{W_n} - f^i_{W_0}\|_1 \to 0 \qquad \textrm{as } n \to \infty \qquad
(i=1,\dots,m).
\]

\noindent Now since
\begin{align*}
\| f(W_n) - f(W_0) \|_1 & \leq \sum_{i=1}^m \|f^i_{W_n} - f^i_{W_0}\|_1,
\end{align*}
it follows that $f(W_n) \to f(W_0)$ as $n \to \infty$. 
The result now follows by Theorem \ref{TclusterGen}.
\end{proof}

\begin{remark}
The assumption that $W'_n$ are uniformly bounded guarantees a well behaved
spectrum in the limit. To illustrate the difficulty of working without
some regularity hypothesis, consider any partition of $[0,1] =
\bigsqcup_{j=1}^\infty I_j$ into countably many measurable subsets with
positive measures. Define the graphon
\[
W_0(x,y) = \sum_{j=1}^\infty {\bf 1}_{I_j \times I_j}.
\]
In this case,
\[
W'_0 = \sum_{j=1}^\infty \frac{1}{\mu_L(I_j)} {\bf 1}_{I_j \times I_j},
\qquad \text{so } \| W'_0\|_2^2 = \sum_{j=1}^\infty 1 = \infty.
\]

\noindent In particular, $W'_0$ is not bounded either.  If we now compute
the eigenvalues and eigenvectors of $W_0'$, we find that the spectrum
consists of just $0$ and $1$, both with infinite multiplicity. In
particular, ${\bf 1}_{I_j}$ is an eigenvector of eigenvalue $1$ for all
$j$.

We can now perturb $W_0$ to $W_1$ while preserving the block structure,
so that there will be one eigenvalue of $W'_1$ in $(1 - \epsilon_j,1)$
corresponding to one eigenvector which is a small perturbation of ${\bf
1}_{I_j}$.  If the $\epsilon_j \to 0$ then the eigenvalues of $W'_1$
converge to $1$.  It is no longer clear how to properly define the
clustering for the limit of the corresponding normalized Laplacian
sequence in that case.
\end{remark}

\appendix

\section{Proofs for dense $S$-colored graph limit theory}\label{Acolor}

\noindent \textbf{Proof of Theorem \ref{LcountingColor}:}

Before proving the result, we explain the general idea of the proof on an
example. We adapt the idea in \cite[Lemma 10.24]{Lovasz:2008}. Suppose
$H$ is a path on $4$ vertices, $V(H) = \{1,2,3,4\}$, $E(H) =
\{(1,2),(2,3),(3,4)\}$. Then,
\begin{align*}
&t_S(H,W) - t_S(H,W') = \\
&\int_{[0,1]^4} f_W(x_1,x_2) f_W(x_2,x_3) f_W(x_3,x_4) \cdot
\prod_{i=1}^4 {\bf 1}_{c_W(x_i) = c_H(i)}\ dx_i \\
&- \int_{[0,1]^4} f_{W'}(x_1,x_2) f_{W'}(x_2,x_3) f_{W'}(x_3,x_4) \cdot
\prod_{i=1}^4 {\bf 1}_{c_{W'}(x_i) = c_H(i)}\ d x_i \\
&=  \int_{[0,1]^4} \left[f_W(x_1,x_2) f_W(x_2,x_3)
f_W(x_3,x_4)-f_W(x_1,x_2) f_W(x_2,x_3) f_{W'}(x_3,x_4)\right] \cdot
\prod_{i=1}^4 {\bf 1}_{c_W(x_i) = c_H(i)}\ dx_i \\
&+ \int_{[0,1]^4} \left[f_W(x_1,x_2) f_W(x_2,x_3)
f_{W'}(x_3,x_4)-f_W(x_1,x_2) f_{W'}(x_2,x_3) f_{W'}(x_3,x_4)\right] \cdot
\prod_{i=1}^4 {\bf 1}_{c_W(x_i) = c_H(i)}\ dx_i \\ 
&+ \int_{[0,1]^4} \left[f_W(x_1,x_2) f_{W'}(x_2,x_3)
f_{W'}(x_3,x_4)-f_{W'}(x_1,x_2) f_{W'}(x_2,x_3) f_{W'}(x_3,x_4)\right]
\cdot \prod_{i=1}^4 {\bf 1}_{c_W(x_i) = c_H(i)}\ dx_i \\
&+ \int_{[0,1]^4} f_{W'}(x_1,x_2) f_{W'}(x_2,x_3) f_{W'}(x_3,x_4)\cdot
\left(\prod_{i=1}^4 {\bf 1}_{c_W(x_i) = c_H(i)}-\prod_{i=1}^4 {\bf
1}_{c_{W'}(x_i) = c_H(i)}\right)\ \prod_{i=1}^4 dx_i
\end{align*}

\noindent Thus, to prove the Counting Lemma, it suffices to obtain a
bound for integrals of the form: 
\[
\int_{[0,1]^4} \left(f_W(x_1,x_2) - f_{W'}(x_1,x_2)\right) \prod_{e \in
E(H) \setminus (1,2)} W_e(x_{e_s},x_{e_t}) \cdot \prod_{i=1}^4 {\bf
1}_{c_W(x_i) = c_H(i)} dx_i, 
\]

\noindent where $0 \leq W_e(x,y) \leq 1$ are arbitrary functions, and a
bound for 
\[
\int_{[0,1]^4} f_{W'}(x_1,x_2) f_{W'}(x_2,x_3) f_{W'}(x_3,x_4)\cdot
\left(\prod_{i=1}^4 {\bf 1}_{c_W(x_i) = c_H(i)}- \prod_{i=1}^4{\bf
1}_{c_{W'}(x_i) = c_H(i)}\right)\ \prod_{i=1}^4 dx_i.
\]
We now provide such bounds. 

\begin{proof}
As explained above, to prove the lemma, it suffices to provide a bound
for 
\begin{equation}\label{Eb1}
\int\displaylimits_{[0,1]^{V(H)}}  (f_W(x_\alpha,x_\beta) -
f_{W'}(x_\alpha,x_\beta))\prod_{e \in E(H) \setminus (\alpha,\beta)}
W_e(x_{e_s}, x_{e_t}) \prod_{v \in V(H)} {\bf 1}_{c_W(x_v) = c_H(v)}
\prod_{v \in V(H)}dx_v
\end{equation}

\noindent where $0 \leq W_e \leq 1$ are arbitrary functions, and a bound
for 
\begin{equation}\label{Eb2}
\int_{[0,1]^{|V(H)|}} \prod_{e \in E(H)} f_{W'}(x_{e_s},x_{e_t})
\left(\prod_{v \in V(H)} {\bf 1}_{c_W(x_v) = c_H(v)}-\prod_{v \in
V(H)}{\bf 1}_{c_{W'}(x_v) = c_H(v)}\right) \prod_{v \in V(H)} dx_v.
\end{equation}

\noindent Given $v \in V(H)$, denote by $\nabla(v) := \{(e_s,e_t) \in
E(H) : e_s = v \textrm{ or } e_t = v\}$, and let
\begin{align*}
f(x_\alpha) &:=  \prod_{e \in \nabla(\alpha) \setminus (\alpha,\beta)}
W_e(e_s,e_t) \prod_{v \in V(H) \setminus \{\alpha,\beta\}} dx_v, \\
g(x_\beta) &:=  \prod_{e \in E(H) \setminus \nabla(\alpha)} W_e(e_s,e_t)
\prod_{v \in V(H) \setminus \{\alpha,\beta\}} dx_v, \\
f_1(x_\alpha) &:= {\bf 1}_{c_W(x_\alpha) = c_H(\alpha)} \prod_{v \in V(H)
\setminus \{\alpha,\beta\}} {\bf 1}_{c_W(x_v) = c_H(v)} dx_v,  \\
g_1(x_\beta) &:= {\bf 1}_{c_W(x_\beta) = c_H(\beta)}. 
\end{align*}

\noindent Using this notation, we obtain the following bound for
\eqref{Eb1}
\begin{align*}
&\left|\int_{[0,1]^{|V(H)|}}  (f_W(x_\alpha,x_\beta) -
f_{W'}(x_\alpha,x_\beta))f(x_\alpha)g(x_\beta)f_1(x_\alpha)g_1(x_\beta)\
\prod_{v \in V(H)} dx_v \right| \\
&\leq \int_{[0,1]^{|V(H)|-2}} \left|\int_{[0,1]^2}(f_W(x_\alpha,x_\beta)
-
f_{W'}(x_\alpha,x_\beta))f(x_\alpha)g(x_\beta)f_1(x_\alpha)g_1(x_\beta)\
dx_\alpha dx_\beta\right| \prod_{v \in V(H) \setminus \{\alpha,\beta\}}
dx_v \\ 
&\leq \int_{[0,1]^{|V(H)|-2}} \boxnorm{W-W'} \prod_{v \in V(H) \setminus
\{\alpha,\beta\}} dx_v \\
&= \boxnorm{W-W'}, 
\end{align*}

\noindent since $0 \leq f(x_\alpha) f_1(x_\alpha) \leq 1$ and $0 \leq
g(x_\beta) g_1(x_\beta) \leq 1$. For \eqref{Eb2}, we have
\begin{align*}
&\left|\int_{[0,1]^{|V(H)|}} \prod_{e \in E(H)} f_{W'}(x_{e_s},x_{e_t})
\left(\prod_{v \in V(H)} {\bf 1}_{c_W(x_v) = c_H(v)}-\prod_{v \in
V(H)}{\bf 1}_{c_{W'}(x_v) = c_H(v)}\right) \prod_{v \in V(H)} dx_v\right|
\\
&\leq \int_{[0,1]^{|V(H)|}} \left|\prod_{v \in V(H)} {\bf 1}_{c_W(x_v) =
c_H(v)}-\prod_{v \in V(H)} {\bf 1}_{c_{W'}(x_v) = c_H(v)}\right| \prod_{v
\in V(H)} dx_v \\
&= \mu_L^{|S|}\left(\left(\cart_{s \in S} c_W^{-1}(s)\right) \Delta
\left(\cart_{s \in S} c_{W'}^{-1}(s)\right)\right) \\
&\leq \sum_{s \in S} \mu_L(c_W^{-1}(s) \Delta c_{W'}^{-1}(s)).
\end{align*}

\noindent The result now follows by a telescoping argument as the one
provided before the proof. 
\end{proof}

\begin{proof}[Proof of Theorem \ref{Tcompact}]
Let $(\W_n)_{n \geq 1} \subset \W_S$ be a sequence of colored graphons.
We will show that $(\W_n)_{n \geq 1}$ has a convergent subsequence. Note
that there exist measure preserving maps $\sigma_n: [0,1] \to [0,1]$ such
that the partitions defined by the $c_{W_n^{\sigma_n}}^{-1}(s)$ for $s
\in S$ are intervals ordered in a fixed arbitrary ordering of colors $S$.
Without loss of generality, we will assume that such transformations have
been applied to the $W_n$ so that the $c_{W_n}^{-1}(s)$ are ordered
intervals. Moreover, since the vector of measures $(\mu_L
(c_{W_n}^{-1}(s)))_{s \in S}$ sits on the simplex which is compact, we
can also assume that this vector also converges as $n \to \infty$. It
follows that the limit $c_0(x) := \lim_{n \to \infty} c_{W_n}(x)$
exists almost everywhere on $[0,1]$ and $\mu_L(c_{W_n}^{-1}(s) \Delta
c_{0}^{-1}(s)) \to 0$ as $n \to \infty$ for every $s \in S$.

The proof now proceeds as in \cite[Theorem 9.23]{Lovasz:2013}. Note that
the original partitions $P_{n,1}$ can always be chosen to respect the
partition defined by $c_{W_n}^{-1}(s)$ for $s \in S$. Since the
successive partitions $P_{n,k}$ are refinements of $P_{n,1}$, they will
also respect the coloring. Proceeding as in \cite[Theorem
9.23]{Lovasz:2013}, we obtain a subsequence $W_{n_k}$ and $W_0 \in \W$
such that $\boxnorm{W_{n_k} - W_0} \to 0$.  Finally, since $c_{W_0} =
c_0$ and $\mu_L(c_{W_{n_k}}^{-1}(s) \Delta c_{0}^{-1}(s)) \to 0$,  then
$\boxnorm{W_{n_k} - W_0}^S \to 0$ as well. This concludes the proof of
the theorem.
\end{proof}

\begin{proof}[Proof of Theorem \ref{thm:Top_Equiv}]
We adapt a proof of L.~Schrijver \cite{Schrijver} to the colored graphon
case.  

Let $\G_S$ be the set of isomorphism classes of $S$-colored finite simple
graphs with no isolated vertices.  Then there is a map $M: (\W_S/\sim,
\delta_\Box^S) \to [0,1]^{\G_S}$ defined by setting the $H$ component of
$M(W)$ to be equal to $t_S(H,W)$.  This map is continuous and
well-defined by the Counting Lemma for colored graphons (Lemma
\ref{LcountingColor}).  Since $(\W_S/\sim, \delta_\Box^S)$ is compact
(Theorem \ref{Tcompact}) and $[0,1]^{\G_S}$ is Hausdorff, it suffices to
show that the map $M$ is injective in order to conclude that it is a
homeomorphism onto its image, thereby concluding the proof.

To show the injectivity of $M$, assume that two colored graphons $U,V \in
\W_S$ have equal homomorphism densities for all $H \in \G_S$.
To show that $\delta_\Box^S(U,V) = 0$ we work with a few sampling
distributions.

Let $\mathbb{H}_S(n,U)$ denote a random weighted graph sampled from $U$
by sampling $(X_i)_{i =1}^n$ i.i.d.~from the uniform distribution on
$[0,1]$, and then taking $U(X_i,X_j)$ to be the weight between nodes $i$
and $j$. The coloring is defined by $c_{\mathbb{H}_S(n,U)}(i) :=
c_U(X_i)$.  Given an $S$-colored weighted graph $H$ with $n$ vertices,
let $\mathbb{G}_S(H)$ denote the finite $S$-colored graph $G$ on $n$
vertices where for $i > j$, $(i,j) \in E(G)$ with probability $H(i,j)$
and $G$ is made symmetric.  The coloring $c_G(i) := c_H(i)$.  Lastly, let
$\mathbb{G}_S(n,U) := \mathbb{G}_S(\mathbb{H}_S(n,U))$, so that the
$\mathbb{H}_S(n,U)$ and $\mathbb{G}_S(n,U)$ are coupled in this way.

Note that $\mathbb{G}_S(n,U) = \mathbb{G}_S(n,V)$ in law for every $n$,
because the probabilities $\bp(\mathbb{G}_S(n,W) = H)$ can be derived
from the homomorphism densities $t_S(H,W)$ by inclusion-exclusion.

By the triangle inequality, $\delta^S_\Box(U,V) \le
\delta^S_\Box(U,\mathbb{G}_S(n,U)) + \delta^S_\Box(V,\mathbb{G}_S(n,U))$.
Thus,
\begin{align*}
\delta^S_\Box(U,V) &\le \E(\delta^S_\Box(U,\mathbb{G}_S(n,U))) +
\E(\delta^S_\Box(V,\mathbb{G}_S(n,U))) \\
&= \E(\delta^S_\Box(U,\mathbb{G}_S(n,U))) +
\E(\delta^S_\Box(V,\mathbb{G}_S(n,V))).
\end{align*}

\noindent To conclude the proof, it therefore suffices to show that
$\E(\delta^S_\Box(\mathbb{G}_S(n,W), W))\to 0$ for any graphon $W \in
\W_S$ as $n \to \infty$. We first show that $\mathbb{H}_S(n,W)$ and
$\mathbb{G}_S(n,W)$ are close when coupled in the obvious way.  Let $H$ be
a weighted graph with $n$ vertices. We claim that there exists some fixed
constant $C > 0$ so that
\begin{equation*}
\bp(d^S_\Box(\mathbb{G}_S(H),H) > \epsilon) \le e^{-\epsilon^2 n^2/C}, 
\end{equation*}

\noindent 
where $d^S_\Box(W,W') = \boxnorm{W-W'}^S$. To bound the
cut-norm of a step function $W$,
\[
\boxnorm{W} = \sup_{A,B \subset [0,1]} \left| \int_{A \times B} W(x,y) dx
dy \right|,
\]

\noindent it suffices to consider only the sets $A,B$ which are composed
of unions of steps.  Therefore we consider for any subsets $A,B \subset
\{1, \dots, n\}$ the random variable
\begin{equation*}
\sum_{i \in A, j \in B} \mathbf{1}((i,j) \in E(\mathbb{G}_S(H))) -
\beta_{ij}(H).
\end{equation*}
The Chernoff Inequality yields
\begin{equation*}
\bp\left(\left|\sum_{i \in A, j \in B} \mathbf{1}((i,j) \in
E(\mathbb{G}_S(H))) - \beta_{ij}(H)\right| > \epsilon n^2 \right) \le 2
\exp\left(\frac{-\epsilon^2 n^4}{C |A| |B|}\right)
\end{equation*}

\noindent for some fixed constant $C > 0$. There are only $4^n$ pairs of
sets $A,B$ so our claim follows by the union bound. We conclude, by
picking $\epsilon = C /\sqrt{n}$, that
\begin{equation*}
\E(d^S_\Box(\mathbb{G}_S(H),H)) \le \frac{C}{\sqrt{n}} + e^{-Cn}
\end{equation*}
which goes to zero as $n \to \infty$.

For graphons $W, W' \in \W_S$, define
\[
d_1^S(W,W') :=  \|W-W'\|_1 + \sum_{s \in S} \mu_L(c_W^{-1}(s)
\Delta c_{W'}^{-1}(s)), 
\]

\noindent and let $\delta^S_1(W,W') := \inf_{\sigma \in S_{[0,1]}}
d^S_1(W,W')$. Clearly, 
\[
\delta^S_\Box(W,W') \leq \delta^S_1(W,W').
\]

\noindent We will now show that $\E(\delta^S_1(\mathbb{H}_S(n,W), W)))
\to 0$.  Let $\mathcal{P}$ be a finite partition of $[0,1]$ which is a
refinement of the fibers of the coloring $c_W$.  Then define
$W_\mathcal{P}$ to be the graphon obtained from $W$ by averaging over the
rectangles defined by the partitions $\mathcal{P}$ with
$c_{W_{\mathcal{P} }} := c_W$.

The triangle inequality yields
\begin{align*}
\E(\delta^S_1(W, \mathbb{H}_S(n,W)))  &\le \delta^S_1(W, W_\mathcal{P})  +
\E(\delta^S_1(W_\mathcal{P},\mathbb{H}_S(n,W_\mathcal{P}))) \\
&+ \E(\delta^S_1(\mathbb{H}_S(n,W),\mathbb{H}_S(n,W_\mathcal{P})))
\end{align*}
where $\mathbb{H}_S(n,W)$ and $\mathbb{H}_S(n,W_\mathcal{P})$ are coupled
by the joint choice of $X_i$ when sampling.

Note that the first term is small for sufficiently fine $\mathcal{P}$.
The second term is small for sufficiently large $n$, since we need only
count the number of points in each partition in $\mathcal{P}$. For the
third term, we claim that
\begin{equation*}
\E(d^S_1(\mathbb{H}_S(n,W),\mathbb{H}_S(n,V))) = d^S_1(W,V)
\end{equation*}
when $\mathbb{H}_S(n,W)$ and $\mathbb{H}_S(n,V)$ are coupled by the joint
choice of $X_i$ when sampling. Indeed, let $X_1, \dots, X_n$ be
independent random variables uniformly distributed on $[0,1]$. Then 
\begin{align*}
\E (\|\mathbb{H}_S(n,W)-\mathbb{H}_S(n,V)\|_1) &= \E \left(\frac{1}{n^2}
\sum_{i,j=1}^n |W(X_i, X_j) - V(X_i, X_j)|\right) \\
&= \frac{1}{n^2} \sum_{i,j=1}^n \int_{[0,1]^2} |W(x_i, x_j) - V(x_i,
x_j)|\ dx_i dx_j \\
&= \|W-V\|_1.
\end{align*}

\noindent To compute the other terms of $d^S_1$, we examine which color
is assigned to each interval $(\frac{i-1}{n}, \frac{i}{n}]$ of the two
graphons. Indeed, for each $s \in S$,
\begin{align*}
\E \left( \mu_L(c_{\mathbb{H}_S(n,W)}^{-1}(s) \Delta
c_{\mathbb{H}_S(n,V)}^{-1}(s)) \right)
&= \frac{1}{n}\sum_{i=1}^n \bp(X_i \in c_W^{-1}(s) \Delta c_V^{-1}(s)) \\
&= \mu_L(c_W^{-1}(s) \Delta c_V^{-1}(s)).
\end{align*}

\noindent We conclude that
\[
\E(d^S_1(\mathbb{H}_S(n,W),\mathbb{H}_S(n,V))) = d^S_1(W,V).
\]

\noindent Finally, by the triangle inequality,
\[ \E(\delta^S_\Box(\mathbb{G}_S(n,W), W))) \le
\E(\delta^S_\Box(\mathbb{G}_S(n,W), \mathbb{H}_S(n,W)))\\
+ \E(\delta^S_1(W,\mathbb{H}_S(n,W))), \]

\noindent and both terms on the right converge to zero as shown above. We
therefore have that
\[
\lim_{n \to \infty} \E(\delta^S_\Box(\mathbb{G}_S(n,W), W))) = 0, 
\]
as desired. This concludes the proof.
\end{proof}

\begin{proof}[Proof of Theorem \ref{thm:density}]
Without loss of generality, assume $S = \{1,\dots, N\}$ and let $(W,c_W)
\in \W_S$. There exist a partition of $[0,1]$ into intervals $I_1, \dots,
I_N$ and a measure preserving bijection $\sigma \in S_{[0,1]}$ such that
$c_W(\sigma(x)) \equiv i$ for a.e.~$x\in I_i$. Without loss of
generality, we will assume $W$ has this property (otherwise, replace $W$
by $W^\sigma$). By the density of graphs in $\W_{[0,1]}$, there exists a
sequence of graphs $(G_n)_{n \geq 1}$ such that $\boxnorm{W_{G_n}-W} \to
0$ as $n \to \infty$. Let $v_n := |V(G_n)|$ and assume without loss of
generality that $V(G_n) = \{1,\dots, v_n\}$. We will also assume, without
loss of generality, that $v_n \to \infty$. Note that for $n$ large
enough, almost every point in the interval $((i-1)/v_n, i/v_n]$ is
contained in one of the $I_j$, say in $I_{J(i)}$. Define $c_{G_n}(i) =
J(i)$. It follows easily that $\boxnorm{(W_{G_n}, c_{W_{G_n}}) - (W,
c_W)}^S \to 0$ as $n \to \infty$. 
\end{proof}

\section{Riesz--Fischer theorem for metric space-valued maps}\label{AppB}

Recall the Riesz--Fischer theorem, which says that $\R^m$-valued $L^p$
functions form a complete (pseudo)metric space. In order to formulate one
of the main results of this paper (Theorem \ref{TclusterGen}) in complete
generality, it is of interest to understand if the Riesz--Fischer theorem
holds for more general spaces, such as Banach spaces or even metric
spaces. We now show the result holds for any complete metric space. We
provide a proof, as we were unable to find it in the literature.

\begin{theorem}[Riesz--Fischer for metric spaces]\label{TRF-metric}
Suppose $(\Omega,\mu)$ is a finite measure space, and $(X,d_X)$ is a
metric space. Given $1 \leq p < \infty$, let $L^p(\Omega,X)$ denote the
Borel-measurable functions $f : \Omega \to X$ such that for any
(equivalently, every) $x \in X$,
\[
\int_\Omega d_X(f(\omega), x)^p\ d \mu < \infty,
\]

\noindent and given $f,g \in L^p(\Omega,X)$, define
\[
d_p(f,g) := \left(\int_\Omega d_X(f(\omega),g(\omega))^p\ d
\mu\right)^{1/p}.
\]

\noindent Now if $(X,d_X)$ is a complete metric space, then $d_p$ equips
$L^p(\Omega,X)$ with the structure of a complete metric space.
\end{theorem}

\noindent As usual, we identify functions in $L^p(\Omega,X)$ that are
equal almost everywhere on $\Omega$.

\begin{proof}
First we reduce the situation to Banach spaces. Fix a point $x_0 \in X$,
and recall that the Kuratowski embedding $\Phi_{x_0} : X \to C_b(X)$
given by
\[
\Phi_{x_0}(x)(y) := d_X(x,y) - d_X(x_0,y)
\]

\noindent is an isometric embedding. Therefore, we may identify $X$ with
its image inside the Banach space $C_b(X)$ via $\Phi_{x_0}$. Now suppose
$f_n$ is a Cauchy sequence in $L^p(\Omega,\Phi_{x_0}(X)) \subset
L^p(\Omega,C_b(X))$. Note that the Riesz--Fischer theorem for maps in
$L^p(\Omega,C_b(X))$ is stated in \cite{Bochner-Taylor}, for instance,
and can be applied to show that $f_n \mapdef{L^p} f$ for some $f \in
L^p(\Omega,C_b(X))$. However, it is not immediate that $f \in
L^p(\Omega,\Phi_{x_0}(X))$, whence we provide a proof for completeness.

The proof follows \cite[Theorem 3.11]{Rudin_Real_Complex}. Since $f_n$
is Cauchy, there exists a sequence of integers $n_1 < n_2 < \cdots$, such
that if $m,n \geq n_k$ then $d_p(f_n,f_m) < 2^{-k}$. Now define
\[
g_k(\omega) := \sum_{j=1}^k d_X(f_{n_j}(\omega), f_{n_{j+1}}(\omega)),
\qquad
g(\omega) := \sum_{j=1}^\infty d_X(f_{n_j}(\omega), f_{n_{j+1}}(\omega)),
\]

\noindent where we will identify $d_X(x,x') = \|
\Phi_{x_0}(x) - \Phi_{x_0}(x') \| =: \| x - x' \|$ for $x,x' \in X$.
Now integrating on $\Omega$, we obtain by Minkowski inequality in
$L^p([0,1],\R)$ (and the choice of $n_k$) that $\| g_k \|_p < 1$ for all
$k$. It follows that $\| g \|_p \leq 1$ by Fatou's Lemma. In particular,
$g(\omega)$ is finite a.e.~$\mu$, whence the series
\begin{align*}
f(\omega) := &\ f_{n_1}(\omega) + \sum_{k=1}^\infty (f_{n_{k+1}}(\omega)
- f_{n_k}(\omega))\\
= &\ \Phi_{x_0}(f_{n_1}(\omega)) + \sum_{k=1}^\infty \left(
\Phi_{x_0}(f_{n_{k+1}}(\omega)) - \Phi_{x_0}(f_{n_k}(\omega)) \right)
\end{align*}

\noindent converges absolutely a.e.~$\mu$. Set $f(\omega) := 0 =
\Phi_{x_0}(x_0)$ on the remaining null set; then $f$ converges absolutely
on all of $\Omega$, hence converges on all of $\Omega$ in the Banach
space $C_b(X)$. Moreover, $f(\omega)$ is the pointwise limit of
$f_{n_k}(\omega) \in \Phi_{x_0}(X) \subset C_b(X)$.
Since $X$ and hence $\Phi_{x_0}(X)$ is complete, it follows that $f$ has
image in $\Phi_{x_0}(X)$.

It remains to show that $f_n \mapdef{L^p} f$ and $f \in
L^p(\Omega,\Phi_{x_0}(X))$. Fixing $\epsilon > 0$, there exists $N$ such
that $\| f_n - f_m \|_p < \epsilon$ for $n,m > N$. Hence if $m>N$, then
by Fatou's Lemma,
\[
\int_\Omega \| f(\omega) - f_m(\omega) \|^p\ d \mu \leq \lim_{\qquad k
\to \infty} \hspace*{-4mm} \inf \int_\Omega \| f_{n_k}(\omega) -
f_m(\omega) \|^p\ d \mu \leq \epsilon^p.
\]

\noindent It follows that $f - f_m \in L^p(\Omega,C_b(X))$, whence $f \in
L^p(\Omega,C_b(X))$ (and hence in $L^p(\Omega,\Phi_{x_0}(X))$ from
above). The preceding computation also shows that $\|f - f_m\|_p \to 0$
as $m \to \infty$, which concludes the proof.
\end{proof}

As an immediate consequence of Theorem \ref{TRF-metric}, we obtain that
continuous metric-valued node-level statistics automatically extend to
$\W_{[0,1]}$.

\begin{cor}\label{Cmetric-ext}
Let $(X,d_X)$ be a metric space and let $f : \G \to L^1([0,1],X)$ be a
continuous node-level statistic (see Definition \ref{Dnode-stat-metric}).
Then $f$ extends to a continuous function $f : \W_{[0,1]} \to
L^1([0,1],X)$.
\end{cor}

Note that the proof of Theorem \ref{TRF-metric} also implies the
following result, which may be interesting in its own right.

\begin{prop}
With $(\Omega,\mu)$ and $(X,d_X)$ as above, every Cauchy sequence in
$L^p(\Omega,X)$ converging to $f \in L^p(\Omega,X)$, has a subsequence
that converges a.e. $\mu$ to $f$. 
\end{prop}

\bibliographystyle{plain}
\bibliography{biblio}

\end{document}